\documentclass[twoside,11pt,reqno]{amsart}
\usepackage{xypic, tensor}
\usepackage{mathtools}
\usepackage[utf8]{inputenc}
\usepackage[normalem]{ulem}

\usepackage[bookmarks,
bookmarksnumbered,%
 colorlinks=true,%
 linkcolor=red,%
 citecolor=blue,%
 filecolor=blue,%
 urlcolor=blue,%
]
{hyperref}

\usepackage{amssymb, amsmath, latexsym}
\usepackage{epigraph}

\makeatletter
\@namedef{subjclassname@2020}{\textup{2020} Mathematics Subject Classification}
\makeatother

\DeclareMathAlphabet{\curly}{OT1}{rsfs}{n}{it}
 
\DeclareMathOperator{\Aut}{Aut}
   
\DeclareMathOperator{\codim}{codim}

\DeclareMathOperator{\rank}{rank}

\DeclareMathOperator{\im}{Im}

\DeclareMathOperator{\card}{Card}

\DeclareMathOperator{\Fix}{Fix}
\DeclareMathOperator{\tr}{tr}

\DeclareMathOperator{\inc}{in}

\DeclareMathOperator{\pr}{pr}
\DeclareMathOperator{\coker}{Coker}

\DeclareMathOperator{\Sq}{Sq}
\DeclareMathOperator{\Bl}{Bl}
\DeclareMathOperator{\Tors}{Tors}
\DeclareMathOperator{\interior}{Int}

\DeclareMathOperator{\Conj}{c}

\DeclareMathOperator{\gr}{gr}
\DeclareMathOperator{\Dim}{dim}

\DeclareMathOperator{\Fl}{Fl}
\DeclareMathOperator{\id}{Id}
\DeclareMathOperator{\Conf}{Conf}

\DeclareMathOperator{\defi}{\mathfrak D}

\def\ra{\rightarrow}
\def\cal{\mathcal} 
\def\wt{\widetilde}

\def\CC{\mathbb{C}}
\def\PP{\mathbb{P}}

\def\ZZ{\mathbb{Z}}
\def\NN{\mathbb{N}}  
\def\RR{\mathbb{R}}

\def\FF{\mathbb{F}}

\def\I{{\mathcal{I}}}

\def\EE{\cal E}
\def\ff{\cal F}

\def\aa{\cal A} 
\def\OO{\cal O} 
\def\BB{\cal B}

\def\mm{\cal M}

\def\s-{\setminus}

\def\HH{\mathbb{H}}

\newtheorem{main}{Theorem}

\newtheorem{thm}{Theorem}[section]
\newtheorem{prop}[thm]{Proposition}
\newtheorem{lemma}[thm]{Lemma}
\newtheorem{defn}[thm]{Definition}

\newtheorem{cor}[thm]{Corollary}
\newtheorem{coro}[main]{Corollary}
\newtheorem{lem}[thm]{Lemma}

\newtheorem{rmk}[thm]{Remark}

\newtheorem{examples}[thm]{Example}

\numberwithin{equation}{section}

\newenvironment{xpl}{\mbox{ }\\{\bf  Examples}\mbox{ }}{
}

\newenvironment{notations}{\mbox{ }\\{\bf  Notations and conventions:}\mbox{ }}

\begin{document}

\title[Kalinin Effectivity and Wonderful Compactifications]
{Kalinin Effectivity and Wonderful Compactifications}

\author[Kharlamov]{Viatcheslav Kharlamov}

\address{
        IRMA UMR 7501, Strasbourg University, 7 rue Ren\'e-Descartes, 67084 Strasbourg Cedex,  FRANCE}

\email{kharlam@math.unistra.fr}

\author[R\u asdeaconu]{Rare\c s R\u asdeaconu}

\address{        
        Department of Mathematics,1326 Stevenson Center, Vanderbilt University, Nashville, TN, 37240, USA}
  \address{
	Institute of Mathematics of the Romanian Academy,  P.O. Box 1-764, Bucharest 014700,  Romania}

\keywords{real algebraic manifolds, Smith exact sequence, Smith-Thom maximality, Galois maximality, Kalinin effectivity}

\subjclass[2020]{Primary: 14P25; Secondary: 14C05, 14J99}

\email{rares.rasdeaconu@vanderbilt.edu}

\begin{abstract} 
We review the definition and main properties of Kalinin effectivity and describe 
methods for constructing effective spaces together with several examples. 
We analyze the Kalinin effectivity of wonderful compactifications and prove that 
the wonderful compactifications of hyperplane arrangements and of configuration 
spaces associated to Kalinin effective compact complex manifolds are themselves 
Kalinin effective. As an application, we show that the Deligne–Mumford space of 
real rational curves with marked points is effective. Finally, we apply Kalinin 
effectivity to study Smith–Thom maximality for Hilbert squares.
\end{abstract}

\maketitle

\thispagestyle{empty}

\setlength\epigraphwidth{.55\textwidth}
\epigraph{I would say that
mathematics is the science of skillful operations with concepts and rules invented 
just for this purpose.}{E.~Wigner - ``The Unreasonable Effectiveness of Mathematics 
in the Natural Sciences" \\}

\setcounter{tocdepth}{2}
\tableofcontents


\section{Introduction}


\subsection{Motivation and Context}


The main object of interest in this paper is the topology of the real loci of 
Deligne-Mumford spaces $\overline{\mathcal M}_{0,n}$ and, more generally, 
the topology of the real loci of so-called wonderful compactifications  \cite{li} of 
various arrangements.

Most of the previously obtained results on the topology of these real manifolds 
concern calculations of their Betti numbers and cohomology rings. In \cite{ehkr}, 
P.~Etingof, A.~Henriques, J.~Kamnitzer, and E.~Rains described 
the integral cohomology ring of the real locus of $\overline{\mathcal M}_{0,n}$ 
for the real model corresponding to a choice of real marked points. For the other 
real structures on $\overline{\mathcal M}_{0,n}$, the additive structure of the 
cohomology with $\FF_2$-coefficients was described by \"O.~Ceyhan \cite{ceyhan}
in terms of graph-homology. More recently, X.~Chen, P.~Georgieva, and A. Zinger 
\cite{cgz} computed  the rational cohomology ring of the Deligne-Mumford moduli 
space of real rational curves with conjugate marked points and showed that the 
integer homology has only $2$-torsion. In the case of the De Concini-Procesi models 
associated to arrangements of real projective spaces, E.~Rains computed in 
\cite{rains} the integral homology, modulo $2$-torsion of their real loci and described 
its multiplicative structure.

Motivated by the importance of Steenrod squares in understanding the topology 
of involutions and their fixed loci, we adopt a different but related point of view. 
Specifically, we apply Kalinin’s approach \cite{kalinin-eff, kalinin}, which relates 
the cohomology ring and the Steenrod algebra structure in equivariant cohomology 
to those of the fixed locus.


\subsection{Framework:  Smith Theory, Kalinin Effectivity, and Conjugation Spaces}


The topology of the real locus of complex manifolds equipped with an 
anti-holomorphic involution has long been studied using the classical methods 
of Smith theory.  A fundamental consequence of this theory is that the total Betti 
number of the fixed locus is bounded above by the total Betti number of the 
ambient complex manifold\footnote{Throughout this paper we will only consider 
cohomology with coefficients in $\FF_2.$}. When the equality is attained, the 
manifold is said to be Smith-Thom maximal. Another important tool in this context is 
the Borel-Serre spectral sequence together with its stabilized version specific to 
involutions, the Kalinin spectral sequence \cite{kalinin-ss, dik}. In this language, 
Smith-Thom maximality is equivalent to the collapse of the Kalinin spectral 
sequence at the first page, while another closely related condition, the Galois 
maximality, corresponds to collapse at the second page 
(see Section \ref{K-effectivity}).

Kalinin's notion of effective space originated from his work \cite{kalinin-ss} on 
the topology of real projective hypersurfaces and was further developed in 
\cite{kalinin-eff}. Informally, effective spaces are spaces equipped with an involution 
for which the limiting term of the Kalinin spectral sequence provides complete 
and functorial control over the cohomology ring and the Steenrod algebra structure 
of the fixed point set (see Theorem \ref{prop-def}). In particular, for Smith-Thom 
maximal effective spaces, this control becomes especially strong: their cohomology 
ring is isomorphic, as an algebra over the Steenrod algebra, to the cohomology 
ring of the fixed point set, as described in Theorem \ref{coh-conj_spaces}.

Kalinin introduced the notion of effective space in his 2002 paper \cite{kalinin-eff}. 
Three years later, J. C.~Hausmann, T.~Holm, and V. Puppe \cite{HHP} introduced 
the notion of conjugation space, with applications ranging from transformation 
groups to toric topology. It turns out, however, that the two theories meet in a very 
precise way: a conjugation space is an effective space in Kalinin’s sense satisfying 
Smith-Thom maximality. We prove this equivalence in Proposition \ref{equivalence}, 
thereby placing conjugation spaces within the framework of Kalinin effectivity.

The class of effective spaces is considerably broader than that of conjugation 
spaces. This additional flexibility makes it possible to analyze the topology of real 
loci in situations where Smith-Thom maximality fails. One of the main objectives 
of this paper is to identify natural geometric constructions that preserve effectivity 
and, when present, also preserve Smith-Thom or Galois maximality. In particular, 
we obtain new examples of Kalinin effective spaces and of Smith-Thom and Galois 
maximal spaces.


\subsection{Main Results and Applications} 


We begin by reviewing Kalinin’s spectral sequence and the notion of Kalinin 
effectivity for pairs, which provides the most convenient framework for the 
applications considered here.We also exhibit several new examples of effective 
spaces. The stability of Kalinin effectivity under blow-ups is then investigated 
and established for pairs satisfying a cohomological condition that we call 
stretchedness (see Section \ref{stretchedness}).

This stability result is subsequently applied to produce new examples of 
Kalinin effective and conjugation spaces among wonderful compactifications 
of arrangements, in the general setting introduced by Li \cite{li}. Our main 
result in this direction, proved in Section \ref{Wond-Compact}, gives a general 
criterion (see Theorem \ref{thm-stretched}) describing when the wonderful 
compactification of an arrangement equipped with a real structure is effective. 
As an application, in Section \ref{examples} we study the Kalinin effectivity of 
several wonderful compactifications of particular interest and obtain the following:


\begin{main}
\label{dcp-theorem}
The De Concini - Procesi wonderful compactification of any arrangement of real linear 
subspaces in $\PP^n$ is a conjugation space. 
\end{main}


The De Concini-Procesi wonderful compactification of the braid arrangement 
$A_{n-2}$ can be identified with the Deligne-Mumford compactification 
$\overline{\mathcal M}_{0,n}$ \cite[page 483]{dp}. This observation shows that the 
method developed for proving Theorem  \ref{thm-stretched} also applies to the moduli 
space of real stable rational curves with real marked points. 

The Deligne-Mumford space $\overline{\mathcal M}_{0,n}$ admits several natural 
real structures $\Conj_\sigma,$ indexed by permutations $\sigma\in S_n$ 
of order $2$ acting on the marked points (see Section \ref{KKspace-effectivity}).
The case $\sigma=\mathrm{id}$ corresponds to real stable rational curves with
all markings real. We prove:


\begin{main}
\label{R-DM-effective}
Let $n\ge 3$ and let $\Conj_\sigma$ be a real structure on the Deligne--Mumford
space $\overline{\mathcal M}_{0,n}$ defined by an order $2$ permutation 
$\sigma\in S_n$. Then the following hold:
\begin{itemize}
  \item[(i)] If $\sigma=\mathrm{id}$, the real pair
  $(\overline{\mathcal M}_{0,n},\Conj_{\mathrm{id}})$ is a conjugation space.
  \item[(ii)] If $\Fix(\sigma)\neq\emptyset$, the real pair
  $(\overline{\mathcal M}_{0,n},\Conj_\sigma)$ is effective and Galois maximal.
\end{itemize}
\end{main}


As a direct consequence of Theorem \ref{R-DM-effective} and Theorem 
\ref{coh-conj_spaces} we obtain the following improvement of Theorem 5.6 in 
\cite{ehkr}:

\begin{coro}
\label{improved-ehkr}
Let $F\overline{\mathcal M}_{0,n}$ denote the fixed locus of the involution 
$\Conj_{\mathrm{id}}$ on $\overline{\mathcal M}_{0,n}.$ Then 
$H^*(F\overline{\mathcal M}_{0,n})$ with its standard and 
$H^{2*}(\overline{\mathcal M}_{0,n})$ with its intrinsic structure
of algebra over Steenrod algebra are naturally isomorphic\footnote{ See Section 
\ref{eff-spaces} for a description of this isomorphism.}.
\end{coro}

Kalinin effective spaces that are not necessarily conjugation spaces also arise
among the wonderful compactifications of configuration spaces of ordered points. 
Let $(X,\Conj)$ be a closed, nonsingular complex variety equipped 
with an anti-holomorphic involution and $n$ a positive integer. Then the space of 
$n$-tuples of distinct points
$$
\Conf_n(X)=\{(p_1,\dots, p_n)\, | \, p_i\in X, \, p_i\neq p_j\, {\text{for every}}\,  i\neq j\}
$$
has a naturally induced anti-holomorphic involution. Several wonderful compactifications 
of $\Conf_n(X)$ were constructed by Fulton-MacPherson \cite{fm}, Ulyanov \cite{ulyanov} 
and Kuperberg-Thurston \cite{kt} and discussed by Li in  \cite{li}. Such compactifications 
automatically satisfy the aforementioned stretchedness condition, and we find:


\begin{main}
\label{eff-wc-config}
Let $(X,\Conj)$ be a closed, nonsingular complex variety equipped with an 
anti-holomorphic involution. 
\begin{itemize}
\item[ 1)] If $X$ is Kalinin effective, then the Fulton-MacPherson, Ulyanov and 
Kuperberg-Thurston wonderful compactifications are Kalinin effective spaces.
\item[ 2)] If $X$ is maximal or Galois maximal, then the Fulton-MacPherson, 
Ulyanov and Kuperberg-Thurston wonderful compactifications
are maximal or Galois maximal spaces, respectively.
\item[ 3)] If $X$ is  a conjugation space, then the Fulton-MacPherson, Ulyanov and 
Kuperberg-Thurston wonderful compactifications are conjugation spaces.
\end{itemize}
\end{main}


As another application of Kalinin effectivity, which was the starting point of our 
investigations of Kalinin effectivity, we compute the deficiency of its Hilbert square 
of effective Galois maximal varieties.


\begin{main}
\label{defect-GM}
Let $X$ be a compact connected nonsingular complex manifold of dimension 
$n,$ such that $H_{\rm odd}(X)=0,$ equipped with a real structure, which is 
effective and Galois maximal.If the Smith-Thom deficiency of $X$ is $\defi(X)=a,$ then 
the Smith-Thom deficiency of the Hilbert square $X^{[2]}$ equipped with the 
induced real structure is 
$$
\defi(X^{[2]})=\sum_{k=1}^{2n}(2k-1)\delta_k+a\beta_*+\frac{a(a-1)}2.
$$
\end{main}


This result improves and extends Theorems 1.7 and 1.8 in \cite{loss-general}. As 
a consequence we find new examples of maximal spaces among Hilbert squares:


\begin{coro} 
\label{square}


The Hilbert square of a maximal effective space  is maximal.
\end{coro}


\subsection*{Acknowledgements}  
The second author acknowledges the support of the University of Strasbourg and 
especially thanks Dan Margalit and Vanderbilt University for their support 
during the preparation of this work.


\notations


\begin{itemize}
\item[1)] By a complex manifold equipped with a real structure we mean a pair 
$(X,\Conj)$ consisting of a complex manifold $X$ together with an anti-holomorphic 
involution $\Conj:X\to X$.

\item[2)]
Let $G=\{\mathrm{id},\Conj\}$ be the cyclic group of order $2$. 
If a complex manifold $X$ is equipped with an anti-holomorphic involution 
$\Conj:X\to X$, then $\Conj$ canonically determines an action of $G$ on $X$, 
and we view $(X,\Conj)$ as a $G$-space with respect to this action.

\item[3)] If $(X,\iota)$ is a space equipped with an involution, we denote 
by $F(X)$ the set of fixed points of $\iota$.

\item[4)] Unless explicitly stated, all homology and cohomology groups are taken 
with coefficients in the field $\FF_2=\ZZ/2\ZZ$. We denote by $\beta_i(\,\cdot\,)$ 
the $i$-th Betti number with $\FF_2$-coefficients, by $\beta_*(\,\cdot\,)$ the 
total Betti number, and by $\beta_{\rm odd}(\,\cdot\,)$ the sum of the Betti 
numbers in odd degrees.

\end{itemize}

\smallskip


\section{Kalinin effectivity and maximality}\label{K-effectivity}


Let $G$ be the cyclic group of order $2.$

\begin{defn}\label{defn-flag}

A {\it $G$-flag} $(V_p,V_{p-1},\dots, V_1;\Conj)$ consists of a CW-complex
$V=V_p$ equipped with a continuous involution $\Conj,$ that sends cells to cells
and acting identically on each invariant cell,  and $\Conj$-invariant 
subcomplexes $V_k\subseteq V$ such that $V_{k}\subseteq V_{k+1}$ for 
all  $k=1,\dots, p-1.$
\end{defn}

For convenience, a $G$-flag will be called a {\it $G$-space} if $p=1$,
a {\it $G$-pair} if $p=2$, and a {\it $G$-triple} if $p=3.$ We will often
omit the involution $\Conj$ from notation of a $G$-flag when it is understood 
from the context.

\medskip

The fixed point sets of a $G$-flag $(V_p,V_{p-1},\dots, V_1;\Conj)$ form a 
flag of $CW$-complexes
$$
F(V_p,\dots, V_1):=(F(V_p),\dots,  F(V_1)),
$$ 
where 
 $F(V_j)=F(V_p)\cap V_j$ is the fixed point set of $V_j,\, j=1,\dots,p.$ 
The flag $F(V_p,\dots, V_1)$ will be considered as a $G$-flag with $G$ acting 
identically on it.


\begin{rmk}
{\rm All $G$-flags considered in this article are flags of manifolds. In what follows, 
we will mainly use  $G$-pairs and $G$-triples.}
\end{rmk}


\subsection{The Kalinin spectral sequence}


\begin{thm}[Kalinin, \cite{kalinin-ss}]
\label{main-ss}
Let $(X,Y)$ be a $G$-pair. Then the following objects exist and are natural 
with respect to $G$-maps:
\begin{itemize}
\item[ 1)] A group filtration $\ff=\{\ff_n(X,Y)\}_{n\geq -1},$ where
$$
0=\ff_{-1}(X,Y)\subseteq\ff_0(X,Y)\subseteq\cdots\subseteq 
\ff_n(X,Y)\subseteq\cdots H^*(F(X,Y)).
$$
\item[ 2)] A $\ZZ$-graded spectral sequence 
$(\prescript{r}{}H^*(X,Y),\prescript{r}{}d^*),$ where 
$$
\prescript{r}{}d^q: \prescript{r}{}H^q(X,Y)\ra \prescript{r}{}H^{q-r+1}(X,Y),
\quad \prescript{r}{}d^{q-r+1}\circ \prescript{r}{}d^q=0,
$$
for which we have
$$
\prescript{1}{}H^q=H^q(X,Y), \quad \prescript{1}{}d^*=1+\Conj^*,
$$
and 
$$
\prescript{2}{}H^q=H^1(G, H^q(X,Y)).
$$
\item[ 3)] An isomorphism
$$
\phi_i:\gr^i_\ff(H^*(F(X,Y)))\ra \prescript{\infty}{}H^i,
$$
where $\gr^i_\ff(H^*(F(X,Y)))$ is the graded group associated to the filtration $\ff.$ \qed
\end{itemize}
\end{thm}


The above spectral sequence is called {\em Kalinin's spectral sequence}.{\footnote{A similar 
spectral sequence in homology exists as well \cite{dik}.} 
We briefly recall how Kalinin's spectral sequence is obtained from the 
Leray-Serre spectral sequence of the Borel fibration.

\medskip

The Kalinin spectral sequence is derived from the Leray-Serre spectral sequence 
$\{E^{p,q}_r\}$ of the Borel fibration $\pi: (X,Y)_G=S^\infty\times_G (X,Y)\to BG$, 
in the following way. Since $BG=\RR\PP^\infty,$ the multiplication by the generator 
$$
u\in H^1(\RR\PP^\infty,\FF_2)\simeq \FF_2[u]
$$ 
produces isomorphisms
$\prescript{r}{}\phi_{p}: \prescript{r}{}E^{p,*}\to \prescript{r}{}E^{p+1, *}$
for $p\gg 0$ and consider the direct limits
$$
\prescript{r}{}H^q:=
\varinjlim (\prescript{r}{}E^{p,q}, \prescript{r}{}\phi_p).
$$ 
By the Borel theorem, the inclusion homomorphisms 
\begin{equation*}
H^i((X,Y)_G)\to H^i(F(X,Y)_G)
\end{equation*}
are isomorphisms for each $i>\dim X$, which establishes natural isomorphisms
\begin{equation}
\label{Borel-identification}
\varinjlim H^i((X,Y)_G)=\varinjlim H^i(F(X,Y)_G)=H^*(F(X,Y)).
\end{equation}
The filtration $\{\ff_n\}$ on $H^*(F(X,Y))$ is defined by 
$$
\ff_n:=\varinjlim (\ff_i'\cap H^{n+i}((X,Y)_G)),
$$ where 
$H^*((X,Y)_G)=\ff'_0\supseteq \ff'_1 \supseteq \dots$ is the filtration
by $\ff'_j=\pi^*(H^{\ge j}(BG))$. In particular, by taking $i$ sufficiently 
large and using the identification (\ref{Borel-identification}),
we obtain 
$$
\ff_n=\ff'_{i-n}\cap H^i((X,Y)_G).
$$

Notice that from functoriality of this spectral sequence it follows that
\begin{equation}
\label{automatic}
\ff_i(X, Y)\subseteq H^{\le i}(F(X,Y))\quad \text{for each $i$}.
\end{equation}

\medskip

For the reader’s convenience, we recall next several basic properties of 
Kalinin's spectral sequence.


\begin{prop}[Kalinin, \cite{kalinin-ss}]
\label{multiplicative-kalinin}
The cup-product in $X$ descends to a multiplicative structure in $\prescript{r}{}H^*,$
so that  $\prescript{r}{}H^*$ is a $\ZZ_2$-algebra. Moreover:
 \begin{itemize}
  \item[ 1)] The differentials $\prescript{r}{}d^*$ are derivations for all $r\geq  2,$ i.e., 
$$
\prescript{r}{}d^*(x\cup y) = \prescript{r}{}d^*x \cup y + x\cup \prescript{r}{}d^*y.
$$
\item[ 2)] The filtration $\ff_*$ is multiplicative, i.e., for every $p,q\geq 0$
$$
\ff_p \cup \ff_q\subseteq \ff_{p+q}.
$$ 
\item[ 3)] The maps  $\phi_i,\, i\geq 0$ satisfy 
$$
\phi_{i+j}(x\cup y)=\phi_i(x)\cup\phi_j(y),
$$
for every $x\in \gr^i_\ff(H^*(F(X,Y)))$ and  $y\in \gr^j_\ff(H^*(F(X,Y))).$
\end{itemize}
\end{prop}


\begin{prop}[Kalinin, \cite{kalinin-eff}]
Let $(X, Y)$ and $(A,B)$ be two $G$-pairs. Then:
\begin{itemize}
\item[ 1)] The multiplication
$$
\times: H^*(X,Y)\otimes H^*(A,B) \to H^*(X\times A, Y\times A \cup X \times B) 
$$
induces spectral sequence isomorphisms
$$
\times: \prescript{r}{}H^i(X,Y)\otimes \prescript{r}{}H^j(A,B) 
\to \prescript{r}{}H^{i+j}(X\times A, Y\times A \cup X \times B),
$$
where 
$
\prescript{r}{}d(a \otimes b) = a \otimes \prescript{r}{}d(b) + 
\prescript{r}{}d(a) \otimes b,
$
for any $r>1.$
\item[ 2)]  We have 
$$
\times: \ff_i(X,Y) \otimes \ff_j(A,B)\subseteq \ff_{i+j}(X\times A, Y\times A \cup X \times B).
$$
\item[ 3)] We have 
$$
\phi_*(a \otimes b) = \phi_*(a) \times  \phi_*(b).
$$
\end{itemize}
\qed
\end{prop}


\begin{prop}[Kalinin, \cite{kalinin-eff}]
\label{al-lef}
Let $(X,Y)$ be a $G$-pair, where $X$ is an $n$-dimensional closed manifold and $Y$ 
a submanifold. Then:
\begin{itemize}
\item[ 1)] For $r\leq \infty,$ the non-degenerate bilinear form
$$
\Phi_{(X,Y)}:H^*(X,Y)\otimes H^*(X \setminus Y)\to \FF_2, \Phi_{(X,Y)}(x\otimes y)=(x\cup y)\cap [X,Y],
$$
induces a non-degenerate bilinear form
$$
\prescript{r}{}\Phi_{(X,Y)}: \prescript{r}{}H^*(X,Y)\otimes \prescript{r}{}H^*(X \setminus Y)\to\FF_2;
$$
for $r<\infty,$ the differentials $\prescript{r}{}d_{(X,Y)}$ and 
$\prescript{r}{}d_{X\setminus Y}$ are mutually adjoint with respect to this form.
\item[ 2)] We have 
$$
\ff_i(X, Y)^{\perp} = w^{-1}\ff_{n-i-1} (X \setminus Y),
$$ 
where $w$ is the total Stiefel-Whitney class of the normal
vector bundle of the submanifold $F(X \setminus Y)$ in $X \setminus Y.$
\item[ 3)] We have 
$$
\prescript{\infty}{}\Phi_{(X,Y)}(\phi_*x\otimes \phi_*y)=\Phi_{F(X,Y)}((
w^{-1}x)\otimes y)
$$
for any $x\in H^*(F(X,Y))$ and $y\in H ^*(F(X\setminus Y)).$
\end{itemize}
\qed
\end{prop}


\begin{prop}[Kalinin, \cite{kalinin}]
\label{boundary}
Let $(X,Y,Z)$ be a $G$-triple. Then:
\begin{itemize}
\item[ 1)] The boundary homomorphism $\delta:H^*(X,Y)\ra H^{*+1}(Y,Z)$ in the 
cohomology sequence of a triple induces
a spectral sequence morphism
$$
\prescript{r}{}\delta: \prescript{r}{}H^*(X,Y)\ra \prescript{r}{}H^{*+1}(Y,Z).
$$
\item[ 2)] $F\delta(\ff_i(X,Y)) \subseteq \ff_{i+1}(Y,Z),$ where $F\delta$ is 
the boundary homomorphism in the cohomology sequence of the triple $F(X,Y,Z)$.
\item[ 3)] $\phi_*F\delta = \prescript{\infty}{}\delta\phi_*.$
\end{itemize}\qed
\end{prop}


\subsection{Effectivity}
\label{eff-spaces}


\begin{defn}
Let $(X,Y)$ be a $G$-pair.
\begin{itemize}
\item[ 1)] The $G$-pair $(X,Y)$ is said to be effective if 
 $$
 \ff_i(X,Y) = \Sq H^{\leq i/2}(F(X,Y)),
 $$ 
 where $\Sq = 1 + \Sq^1 +\Sq^2 +\cdots$  is the total Steenrod operation. 
 \item[ 2)] The  $G$-space $X$ is said to be effective if the 
 $G$-pair $(X,\emptyset)$ is effective.
 \end{itemize}
 \end{defn}


\begin{prop}
\label{rel=PAL}
Let $(X,\Conj)$ be a $G$-space, where $X$ is an $n$-dimensional 
closed manifold, and $Y$ a $\Conj$-invariant submanifold. Then, the 
$G$-pair $(X,Y)$ is effective if and only if the $G$-space 
$X\setminus Y$ is effective.
\end{prop}


\proof Proposition \ref{al-lef} and the Wu formulas imply that
$$
(\Sq^{-1} \ff_i(X,Y))^{\perp} = \Sq^{-1} \ff_{n-i-1}(X \setminus Y),
$$
where $\Sq = 1 + \Sq^1 +\Sq^2 +\cdots$  is the total Steenrod operation. 
The proof of the assertion follows now directly from the definition and the 
Poincar\'e-Alexander-Lefschetz duality.
\qed

 Let $(X, Y)$ be an effective $G$-pair. We denote by 
 $$
 \psi^i_{X,Y}: H^i(F(X,Y))\to \prescript{\infty}{}H^{2i}(X,Y).
 $$

 the composition of the homomorphisms
$$
H^i(F(X,Y))\xrightarrow{\Sq} \ff_{2i}(X,Y)\xrightarrow{\pr}
\frac{\ff_{2i}(X,Y)}{\ff_{2i-1}(X,Y)}\xrightarrow{\prescript{\infty}{}\phi_{2i}}
\prescript{\infty}{}H^{2i}(X,Y).
$$

The following remarkable property was one of the main Kalinin's 
motivations for introducing and study of the notion of effective spaces.


\begin{thm}[Kalinin \cite{kalinin-eff, kalinin}]
\label{prop-def}
 Let $(X, Y)$ be an effective $G$-pair. Then: 
\begin{itemize}
\item[ 1)] $\psi^i_{X,Y}$  is an isomorphism.
\item[ 2)] The homomorphism
$$
\psi^*_{X,Y}=
\bigoplus_{i\geq 0}\psi^i_{X,Y}: H^*(F(X,Y))\rightarrow 
\prescript{\infty}{}H^*(X,Y)
$$
is a ring isomorphism.
\item[ 3)] $\psi^*_{X,Y}(\Sq^i x)=\Sq^{2i} (\psi^*_{X,Y}(x))$ for any $x\in H^*(F(X,Y))$. \qed
\end{itemize}
\end{thm}


\begin{prop}{\rm (\cite[Proposition 3.6]{kalinin})}
\label{delta-vanish}
Let 
$(X,Y,Z)$ be a $G$-triple such that the $G$-pairs
$(X,Y)$ and $(Y,Z)$ 
are effective. Then the boundary homomorphism
$$
\delta:H^*(F(Y,Z))\ra H^{*+1}(F(X,Y))
$$
is trivial.
\end{prop}


\proof
Let $x\in H^i(F(Y,Z)).$ Using item 2) of Proposition \ref{boundary}, we find
$$
\Sq(\delta x)=(F\delta)(\Sq x) \in \ff_{2i+1}(X,Y)\subseteq \ff_{2i+2}(X,Y).
$$ 
Therefore,
$$
\psi^{i+1}_{X,Y}(\delta x)= \phi_{2i+2}(\pr(\Sq(\delta x)))=\phi_{2i+2}(0)=0.
$$

Since the pair $(X,Y)$ is effective, the map $\psi^{i+1}_{X,Y}$ is an isomorphism, 
and so $\delta x=0.$
\qed


\subsection{Maximality}


Let $(X,Y;\Conj)$ be $G$-pair, where $\dim\, H^*(X,Y)<\infty,$ and denote  
by $\bar X$ and $\bar Y$ the quotient spaces $X/\Conj$ and $Y/\Conj,$ 
respectively. There exists \cite[Section 1.2.8]{dik} a natural exact sequence

\begin{align}
\label{rses}
\cdots \ra H^{k}(\bar X, F(X)\cup \bar Y)\xrightarrow[]{\pr^*} 
&H^k(X,Y)\xrightarrow[]{\tr_*\oplus\inc^*}\notag\\
H^k(\bar X,F(X)\cup \bar Y)\oplus H^k(F(X), F(Y)) 
&\xrightarrow{\Delta_k} H^{k+1}(\bar X,F(X)\cup \bar Y)\xrightarrow[]{\,}\cdots, 
\end{align}
and a similar one in homology. The connecting homomorphisms $\Delta_k$ 
are given by
$
x\oplus f\mapsto x\cup\omega+\delta x,
$
where $\omega \in H^1(\bar X\setminus F)$ is the characteristic class of 
the double covering $X\setminus F\to\bar X\setminus F.$ 
The sequence (\ref{rses}) is called the {\it Smith sequence} in cohomology. 

As a consequence, the following inequalities hold:
\begin{itemize}
\item[ 1)] $\dim\, H^*(F(X,Y))\le \dim\, H^*(X, Y)$ ({\it Smith inequality}).
\item[ 2)] $\dim\, H^*(F(X,Y))\le \dim\, H^1(G; H^*(X, Y))$ ({\it Borel-Swan inequality}),
\end{itemize}
where $H^1(G; H^*(X, Y))=\ker(1+\Conj_*)/\im(1+\Conj_*).$
The difference
$$
\defi(X,Y)= \dim\, H^*(X,Y)-\dim\, H^*(F(X,Y))
$$ 
is called {\it the Smith-Thom deficiency of} $(X,Y).$ We will use the simplified 
notation $\defi(X)$ to denote $\defi(X,\emptyset).$ 
\begin{defn}
\label{def-max}
Let $(X,Y;\Conj )$ be a $G$-pair, where $\dim\, H^*(X,Y)<\infty.$
\begin{itemize}
\item[ 1)] The pair $(X,Y)$ is called {\it Smith-Thom maximal}
if its Smith-Thom deficiency vanishes.
\item[ 2)] The pair $(X,Y)$ is called {\it Galois maximal}, or a {\it $GM$-pair}, if 
 $$
 \dim\, H^*(F(X,Y))=\dim\, H^1(G; H^*(X, Y)).
 $$
\end{itemize}
\end{defn}


\begin{rmk}
\label{terminology}
{\rm
Throughout this article, for convenience, by ``maximal'' we will mean
``Smith-Thom maximal", whenever the context allows.
}
\end{rmk}


\smallskip

The Smith-Thom or Galois maximality conditions control the stage at 
which Kalinin's spectral sequence degenerates and therefore quantify 
how closely the cohomology of the fixed locus reflects that of the 
ambient space. The following criteria are consequences of Theorem 
\ref{main-ss} ({\it cf.} \cite{dik} and \cite{krasnov}, respectively).


\begin{prop}
\label{GM-def}
Let  $(X,Y)$ be a $G$-pair.
\begin{itemize}
\item[ 1)] The pair $(X,Y)$ is a maximal pair if and only if Kalinin's spectral sequence 
collapses at the first page.
\item[ 2)] The pair $(X,Y)$ is  a Galois maximal pair if and only if Kalinin's spectral 
sequence collapses at the second page. \qed
\end{itemize}
\end{prop}


We recall next the following criteria of maximality and Galois maximality, 
respectively.


\begin{prop}\cite{AP, Fr}
\label{coh-max}
Let $(X,Y)$ be a $G$-pair. The following conditions are equivalent:
\begin{itemize}
\item[ 1)] $(X,Y)$ is a maximal pair.
\item[ 2)] The action of $G$ on 
$H^{*}(X,Y)$  is trivial and the Leray-Serre spectral sequence 
associated with the fibration $(X,Y)\hookrightarrow (X_G, Y_G) \ra BG$
degenerates at the second page.
\item[ 3)] $H^*_G(X,Y)$ is a free $H^*(BG)=H^*(\RR\PP^\infty)$-module.
\item[ 4)] The restriction homomorphism 
$\rho : H^{*}_{G}(X,Y) \rightarrow H^{*}(X,Y)$ is surjective.
\item[ 5)] The restriction homomorphism 
$r_G:H^{*}_{G}(X,Y) \rightarrow H^{*}_G(F(X,Y))$ 
is injective.
\qed
\end{itemize}
\end{prop}


\begin{prop}
\label{GM-char}
Let $(X,Y)$ be a $G$-pair. The following conditions are equivalent:
\begin{itemize}
\item[ 1)] $(X,Y)$ is a Galois maximal pair.
\item[ 2)] Any $\Conj_*$-invariant homology class $\alpha\in H_*(X,Y)$  
can be represented by a $\Conj$-invariant cycle.
\end{itemize}
\end{prop}


\proof 
This is a known easy consequence of the Smith exact sequence in 
homology (see \cite[item 1.2.4]{dik}).
\qed


\begin{prop}
\label{Image-Beta}
If a $G$-pair $(X,Y)$ is Galois maximal, then Kalinin's filtration
$\{\ff_n\}$ coincides with the filtration of $H^*(F(X,Y))$ by the images 
$\im \beta^n$ of the homomorphisms
$$
\beta^n : H^n_G(X,Y)\to H^n_G(F(X,Y))=
\bigoplus_{i\leq n} H^{i}(F(X,Y)).
$$
\end{prop}


\proof Since the pair $(X,Y)$ is Galois maximal, the Borel-Serre spectral 
sequence degenerates at the second page, so that 
\begin{align*}
\prescript{\infty}{}E^{p,q}=\prescript{2}{}E^{p,q}=&\, H^1(G;H^q(X,Y)),\,  p>0,q\ge 0,\\ 
\prescript{\infty}{}E^{0, q}=\prescript{2}{}E^{0, q}=&\, H^q(X,Y)^G.
\end{align*} 
Furthermore, for every $r\geq 0,$ the ring homomorphism 
$$
H^*_G(X,Y)\to H^{* +1}_G(X,Y)
$$ 
given by multiplication by the generator $u$ of $H^*_G(pt)=\FF_2[u]$ 
acts on the spectral sequence as the simple shift 
$$
\prescript{r}{}E^{p,q}=H^1(G,H^q(X,Y))\to \prescript{r}{}E^{p+1,q}=H^1(G,H^q(X,Y))
$$  
if $p>0$ and as the quotient epimorphism 
$$
\prescript{r}{}E^{0,q}=H^q(X,Y)^G\to \prescript{r}{}E^{1,q}=H^1(G,H^q(X,Y))
$$
with $ \im \{1+\Conj^* :  H^q(X,Y)\to H^q(X,Y)\subseteq H^q(X,Y)^G\}$
as the kernel. This implies that 
$$
\beta^n(H^n_G(X,Y))=\beta^{n+i}(H^{n+i}_G(X,Y)\cap \ff'_n)=\ff_n
$$ 
for every $n,i\ge 0$.
\qed

We conclude this section with several results concerning the 
cohomology of maximal Kalinin effective spaces. Let $(X,Y)$ 
be a Smith-Thom maximal Kalinin effective $G$-pair.


\begin{prop}
\label{no-odd}
The odd dimensional cohomology groups $H^{2i+1}(X,Y)$, $i\ge 0$, 
vanish for any Smith-Thom maximal Kalinin effective $G$-pair.
\end{prop}


\proof 
According to Proposition \ref{GM-def} and Theorem \ref{prop-def}, if 
$(X,Y)$ is a maximal pair, then $\ff_{r}(X,Y)$ is isomorphic to 
$H^{\le r}(X,Y)$ for any $r\geq 0.$ It remains to notice that, due to 
the definition of an effective pair, 
$$
\ff_{2i+1}(X,Y)=\Sq H^{\leq i}(F(X,Y))=\ff_{2i}(X,Y)
$$ 
for any $i\geq 0$.
\qed

As a consequence of Proposition \ref{no-odd}, the odd Steenrod squares 
of $(X,Y)$ vanish. This implies that $H^{2*}(X,Y)$ inherits an {\it intrinsic} 
structure of an algebra over (mod $2$) Steenrod algebra with the action of 
the $i$-th standard generator of the Steenrod algebra by $\Sq^{2i}$ 
(instead of $\Sq^{i}$). On the other hand, applying Theorem \ref{prop-def}, 
under assumptions of Proposition \ref{no-odd} we have a ``doubling" 
ring isomorphism 
$$
\psi^*_{X,Y}=
\bigoplus_{i\geq 0}\psi^i_{X,Y}: H^*(F(X,Y))\rightarrow 
\prescript{\infty}{}H^{2*}(X,Y)=H^{2*}(X,Y)
$$
satisfying
\begin{equation}
\label{steenrod-compatibility}
\psi^*_{X,Y}(\Sq^i x)=\Sq^{2i} (\psi^*_{X,Y}(x)),
\end{equation}
for any $x\in H^*(F(X,Y)).$ 
Thus, we get:

\begin{thm}
\label{coh-conj_spaces}
Let $(X,Y)$ a Smith-Thom maximal effective $G$-pair. Then, $H^{2*}(X,Y)$ 
with its intrinsic structure of algebra over Steenrod algebra is naturally 
isomorphic to $H^*(F(X,Y))$ with its standard structure of algebra over 
Steenrod algebra.
\qed
\end{thm}

\begin{rmk}
{\rm The conclusion of Theorem \ref{coh-conj_spaces} was one of motivations 
for Kalinin's introduction of the concept of effective spaces in 2002 
(see Remark 2.2 in \cite{kalinin}). The key result (\ref{steenrod-compatibility}) 
was rediscovered later, in 2006, by M.~Franz and V.~Puppe \cite{fp} in the context 
of conjugation spaces. 
}
\end{rmk}


\subsection{Conjugation Spaces}


\begin{defn}[Hausmann-Holm-Puppe \cite{HHP}]
A $G$-pair $(X, Y)$ is a conjugation $G$-pair if the following conditions 
are satisfied:
\begin{itemize}
\item[ 1)] $H^{\rm odd}(X,Y)=0,$
\item[ 2)] There exist a section $\sigma: H^*(X,Y)\to H^*_G(X,Y)$ of 
the restriction homomorphism $\rho : H^*_G(X,Y)\to H^*(X,Y)$ together 
with a degree-halving isomorphism $\kappa : H^{2*}(X, Y)\to H^*(F(X, Y))$ 
with the property that for every $x\in H^{2n}(X, Y), n\in \NN$ there exist 
elements $y_1,\dots, y_n\in H^*(F(X, Y))$ such that
$$
r_G (\sigma(x))= \kappa(x)u^n+y_1u^{n-1}+\dots +y_n \in H^*_G(F(X,Y))=H^*(F(X, Y))\otimes \FF_2[u],
$$
where $r_G: H^*_G(X, Y)\ra H^*_G(F(X, Y))$ is the restriction homomorphism
in equivariant cohomology.
\end{itemize}
\end{defn}

Furthermore, as it is proved in \cite{fp}, $\sigma$ and $\kappa$ are unique, 
multiplicative, and $y_i=\Sq^i (\kappa(x))$ for each $i=1,\dots, n.$


\begin{prop}
\label{equivalence} 
$(X, Y)$ is a conjugation $G$-pair if and only if $(X, Y)$ is a maximal and effective $G$-pair.
\end{prop}


\proof 
Recall, first, that $(X, Y)$ is a maximal pair if and only if
$H^*_G(X, Y)$ is a free $H^*(BG)=H^*(\RR\PP^\infty)$-module 
$H^*(\RR\PP^\infty)\otimes H^*(X, Y)=H^*(X, Y)[u]$
(see Proposition \ref{coh-max}).
Thus, if a $G$-pair $(X, Y)$ is maximal and effective, 
the inverse of Kalinin's isomorphism $\psi^*$ can be taken as 
a halving isomorphism $\kappa.$ Moreover, since 
substituting $u=1$ into
$$
r_G : H^*(X, Y)[u]\to H^*(F(X, Y))[u]
$$ 
transforms the filtration $H^{\le i}(X, Y)\subseteq H^*(X, Y)$ 
into Kalinin's filtration $\ff_i(X, Y) = \Sq H^{\leq i/2}(F(X, Y))$, t
he required homomorphism  $\sigma$ becomes well defined
by an equation 
\begin{equation}\label{conj-eqtn}
r_G(\sigma(x))= \kappa(x)u^n+\Sq^1(\kappa(x)) u^{n-1}+\dots +\Sq^n(\kappa(x)).
\end{equation}
Conversely, if $(X, Y)$ is a conjugation $G$-pair, then due 
to existence of a halving isomorphism it is maximal, while due 
to \cite{fp} the conjugation equation is necessary of the form 
(\ref{conj-eqtn}), which implies Kalinin's effectivity condition 
$\ff_i(X, Y) = \Sq H^{\leq i/2}(F(X, Y))$.
\qed

\medskip

To conclude this section, we recall without proof a result of Kalinin  
\cite[Lemma 3.10]{kalinin} which we will use in the next section.
\begin{lem}
\label{max}
Let $(Y,Z)$ be an effective maximal pair, and let $Z$ be an effective 
maximal space. Then:
\begin{itemize}
\item [1)] $Y$ is a maximal space.
\item[ 2)] There exists a short exact sequence
$$
0\ra \prescript{\infty}{}H^*(Y,Z) \ra \prescript{\infty}{}H^*(Y)\ra \prescript{\infty}{}H^*(Z)\ra 0.
$$
\qed
\end{itemize}
\end{lem}


\section{Examples and constructions}
\label{examples_and_constructions}


\subsection{Examples of Kalinin effective spaces}


Examples of effective $G$-spaces can be constructed by successively 
attaching {\it $G$-cells}.

\begin{defn}
A \emph{$G$-cell of complex dimension $n$} is the pair $(\CC^n, \mathrm{conj})$
where $\mathrm{conj}$ is the standard complex conjugation.
\end{defn}


\begin{defn}{\rm (}cf. \cite{kalinin, HHP}{\rm )}
We say that a $G$-space $(X,\Conj)$ admits a $G$-{\it cellular decomposition} 
if there exists a finite $G$-flag $(A_r,\dots,A_0)$ 
such that:
\begin{itemize}
\item[ 1)]  For every $i\geq 0,\, A_i\subseteq X$ is a closed subset,  and 
$A_{0}=\emptyset;$ 
\item[ 2)] For every $i\geq 1,$ the $G$-space $(A_i \setminus A_{i-1},\Conj)$ 
is $G$-equivariantly homeomorphic to a $G$-cell.
\end{itemize}
\end{defn}


\begin{lem}
\label{eff-add-cell} 
If $(X,Y)$ is a $G$-pair such that $Y$ is an effective space and 
$X\setminus Y$ is a  $G$-cell, 
then $X$ is an effective space.
\end{lem}


\proof Since $X\setminus Y$  is homeomorphic to a $G$-cell, it is 
effective. Hence, by Proposition \ref{rel=PAL}, the $G$-pair $(X,Y)$ is 
effective, and, according to Proposition \ref{delta-vanish}, for each 
$i\ge 0,$ we have a short exact sequence
\begin{equation*}
0\to \ff_i(X,Y)\to \ff_i(X)\to \ff_i(Y)\to 0.
\end{equation*}
Therefore, we get a short exact sequence
$$
0\to \Sq^{-1}\ff_i(X,Y)\to \Sq^{-1}\ff_i(X)
\to  \Sq^{-1}\ff_i(Y)\to 0.
$$
From the effectivity of $(X,Y)$ and $Y$, we have
$$
\Sq^{-1}\ff_i(X,Y)=H^{\le \frac{i}2}(F(X),F(Y)) \quad{\text{and}}\quad 
\Sq^{-1}\ff_i(Y)=H^{\le \frac{i}2}(F(Y))
$$
Since the inclusion and relative homomorphisms respect the 
dimension grading, we find that
$$
\Sq^{-1}\ff_i(X)=H^{\le \frac{i}2}(F(X)),
$$
and so $X$ is effective.
\qed


\begin{thm}
\cite[Theorem 1.2]{kalinin}
\label{eff-cellular}
Let $(X,\Conj)$ be a $G$-space admitting a $G$-cellular decomposition\footnote{Kalinin 
states this result in \cite{kalinin} for real algebraic varieties, but his proof  does not use 
the algebraicity assumption. This theorem can also be used to produce non-algebraic 
examples of conjugation spaces. One such example is $S^6$ equipped with the standard 
almost-complex structure and its naturally associated real structure.}.
Then $(X,\Conj)$ is a conjugation space.
\end{thm}


\proof
The proof follows immediately by induction on the number of 
$G$-cells, from Proposition \ref{equivalence}, Lemma \ref{max} and 
Lemma  \ref{eff-add-cell}.
\qed

\begin{cor}
\label{cell-examples}
Real projective spaces, real Grassmannians, real flag varieties, and nonsingular 
toric varieties are conjugation spaces.
\end{cor}

Further examples of conjugation spaces are obtained as a direct consequence of 
the following result of van Hamel \cite{hamel}:


\begin{thm}
\label{hamel-conj}
Let $(X,\Conj)$ be a compact, complex manifold equipped with an 
anti-holomorphic involution,such that $H^{*}_{\rm odd}(X)=0.$ If each 
element of $H^*(X)$ can be represented by a $\Conj$-invariant 
complex analytic cycle, then $X$ is a conjugation space.
\qed
\end{thm}


Examples of effective spaces that are not conjugation spaces can be obtained 
as a consequence of the following results:


\begin{prop}
\label{hyperquadrics} 
Every real nonsingular quadric hypersurface is an effective space.
\end{prop}


\proof
Even-dimensional maximal quadric hypersurfaces $X$ admit a real 
$G$-cellular decomposition, and thus for them the result follows from 
Theorem \ref{eff-cellular}. Otherwise,  there exists an integer $s$ 
with $0\le s < \frac12\dim X$ such that  $H^*(F(X))$ is generated 
by the duals of the fundamental classes of iterated hyperplane sections 
$[F(X)\cap F(H^r)]$ with $0\leq r\leq s$ and  the duals of the fundamental 
classes of linear subspaces $[F(L^r)]$ with $0\leq r\leq s$. Since the 
Poincar\'e duality descends to Kalinin's spectral sequence (see \cite{dik}), 
the classes $[X\cap H^r]$ and $[L^r]$  yield non-zero classes in 
$\prescript{\infty}{}H^*(X).$ In particular, we find $\ff_{2i}(X)=\ff_{2i+1}(X)$ 
for each $i$, $\ff_{2i}(X)/\ff_{2i-1}(X)=\FF_2$ for each $i\le r$ and each $i\ge n-r,$ 
while $\ff_{2r}(X)=\dots=\ff_{2n-2r-1}(X)$. It remains to observe that
for each of the classes $u=[Z]^*\in H^*(X),$ we have $\phi^{-1}(u)=\Sq([F(Z)]^*),$ 
where $Z$ is any of the submanifolds $X\cap H^r$ or $L^r,\,0\leq r\leq s$, which 
follows from  \cite[Corollary A.2.10]{dk}. Thus, every graded piece 
$\ff_{2i}(X)/\ff_{2i-1}(X)$ is generated by Steenrod squares of elements in 
$H^{\leq i/2}(F(X)),$  and hence $X$ is effective by definition.
\qed


\begin{prop}
\label{effective-GM}
If $X$ is either a nonsingular real curve with nonempty connected real locus or 
a nonsingular real surface $X$ with nonempty connected real locus and $H^1(X)=0,$ 
then $X$ is effective and Galois maximal. 
\end{prop}


\proof
When $F(X)$ is nonempty and connected, we find $\dim\, \ff_0(X)=1.$ 
Since $\ff_0(X)\subseteq H^0(F(X)),$ we conclude that 
$\ff_0(X)=H^0(F(X)).$

If $X$ is a curve with a nonempty connected  real locus, 
then $\prescript{\infty}{}H^i(X)=\FF_2$ for $i=0,2$ and it 
is $0$ for $i=1.$ Therefore, in this case, $\ff_0(X)=H^0(F(X))$ 
is the only nontrivial term in Kalinin's filtration. This implies 
effectivity in the case of curves with a connected real locus.

If $X$ is a surface with $H^1(X)=0$, we find $\prescript{\infty}{}H^i(X)=\FF_2$ 
for $i=0,4$ and it is $0$ for $i=1,3$. By the same arguments 
as above we get 
\begin{align*}
\ff_0(X)=&\,H^0(F(X)),\\
\ff_1(X)=&\, \ff_0(X),\\
\ff_3(X)= &\, \ff_2(X),\\ 
\ff_4(X)=&\, 
H^*(F(X)),\\
\ff_4(X)/\ff_3(X)=&\, \FF_2.
\end{align*}
Finally, to prove the effectivity of $X$ there remain to check that 
$\ff_2^\perp\subset H_*(F(X))$ is generated by $W_1$ and $[F(X)]$ 
where $W_1\in H_1(F(X))$ is the Poincar\'e dual to $w_1(F(X))$ 
and $[F(X)]\in H_2(F(X))$ is the fundamental class. The latter 
property is a direct consequence of \cite[Proposition 2.4.9(1)]{dik} 
and the duality between Kalinin's spectral sequences in homology 
and cohomology (see \cite[Subsection 1.5.3 ]{dik}).

The Galois maximality in both cases discussed above holds in 
more generality, without assumption of connectedness of the real 
locus (see also  \cite[page 262]{krasnov}). The assumption 
$F(X)\ne\emptyset$ and, in the case of surfaces, $H^1(X)=0$ 
imply, respectively,  $\prescript{r}{}H^0(X)=\FF_2$ and, in the 
case of surfaces,  $\prescript{r}{}H^1(X)=0$  for any $r\ge 1$. 
Then, by duality (see Proposition \ref{al-lef} (1)), $\prescript{r}{}H^2(X)=\FF_2$
for curves, and $\prescript{r}{}H^3(X)=0,\, \prescript{r}{}H^4(X)=\FF_2$ 
for surfaces. As a consequence, $\prescript{r}{}d^q=0$ for any 
$r\ge 2, q\ge 0$, which is equivalent to Galois maximality according 
to Proposition \ref{GM-def}.
\qed

\medskip

We conclude this subsection by recalling, without proof, the 
following result of Kalinin, which is one of the main ingredients 
in the proof of Theorem \ref{defect-GM}:


\begin{prop}\cite[Proposition 1.3]{kalinin}
\label{sym-kalinin}
Let $X$ be a finite CW-complex and let 
$\tau : X \times X\ra X \times X,\, \tau(x,y)=(y,x)$.
Then $(X \times X, \tau)$ is a Galois maximal, effective 
space.\qed
\end{prop}


\subsection{Kalinin effectivity under standard constructions}
\label{K_eff_constructions}


We now turn to a systematic construction of further examples of effective spaces, 
and proceed by recalling the following two results of Kalinin:


\begin{prop}\cite[Proposition 3.3]{kalinin}
\label{products-effective}
A finite product of effective pairs is an effective pair.\qed
\end{prop}


\begin{prop} \cite[Proposition 3.11]{kalinin}
\label{effective-bundles}
Let $(X,\Conj)$ be a $G$-space, $\xi$ a complex vector bundle on $X$ endowed with 
an anti-linear involution covering the involution $\Conj,$ and  $Y$ an associated bundle 
of flags in $\xi$ equipped with the induced involution.
If $X$ is effective then $Y$ is effective. \qed
\end{prop}


\subsubsection{Effective blowing-up} 
\label{ebu}


Let $(Y,B, \Conj)$ be a $G$-pair, where $Y$ is a complex manifold 
equipped with an anti-holomorphic involution $\Conj$, and $B\subseteq Y$ 
is a $\Conj$-invariant complex submanifold of complex codimension at least 
two. Consider the blow-up  $\pi: X\ra Y$ of $Y$ along $B,$ and denote by 
$E$ the exceptional divisor. Notice that $(X,E)$ is  a $G$-pair when equipped 
with the induced anti-holomorphic involution.


\begin{prop}
\label{blowup-emptyset}
If $Y$ is effective and $F(B)=\emptyset,$ then $X$ is effective.
\end{prop}


\proof
Since
$F(B)=\emptyset $,  the pull-back homomorphism $\pi^* : H^*(Y)\to H^*(X)$
induces an isomorphism between the second pages of Kalinin's spectral 
sequences for $X$ and $Y$. Therefore, due to naturality of Kalinin's spectral 
sequences, it induces isomorphisms up to the limit. Since, in addition, $\pi$ 
establishes a homeomorphism between $F(Y)$ and $F(X)$, by the naturality 
of Kalinin's filtration, $\pi^* : H^*(F(Y))\to H^*(F(X))$ induces an isomorphism 
between Kalinin's filtrations in $F(Y)$ and $F(X)$. Hence, the effectivity of $Y$ 
is equivalent to that for $X$.
\qed 


\begin{prop}
\label{effectivity-blowup}
If $B$ and $(Y,B)$ are effective, then  $E, \,X$ and $(X,E)$ are effective. 
\end{prop}


\proof
By Proposition \ref{effective-bundles}, we find that the effectivity of $B$ implies the 
effectivity of $E.$ On the other hand, from Proposition \ref{rel=PAL} we see that 
the effectivity of the pair $(Y,B)$ is equivalent to effectivity of the pair $(X,E).$ 
As a consequence of Proposition \ref{delta-vanish}, we find that  the boundary 
homomorphisms
$\delta_\RR^i: H^i(F(E))\to H^{i+1}(F(X),F(E))$
are trivial for every $i\geq 0$. Hence, for each $i\geq 0$ we have a short exact 
sequence 
$$
0\to \ff_i(X,E)\to \ff_i(X)\to \ff_i(E)\to 0.
$$
Therefore, we get a short exact sequence
$$
0\to \Sq^{-1}\ff_i(X,E)\to \Sq^{-1}\ff_i(X)
\to  \Sq^{-1}\ff_i(E)\to 0.
$$
From the effectivity of  $(X,E)$ and $E$, we have
$$
\Sq^{-1}\ff_i(X,E)=H^{\le \frac{i}2}(F(X),F(E)) \quad{\text{and}}\quad 
\Sq^{-1}\ff_i(E)=H^{\le \frac{i}2}(F(E))
$$

Since the inclusion and relative homomorphisms respect the dimension grading, 
we find that
$$
\Sq^{-1}\ff_i(X)=H^{\le \frac{i}2}(F(X)),
$$
and so $X$ is effective.
\qed


\subsection{Constructions of maximal spaces}


We discuss next the behavior of maximal and Galois maximality, respectively, 
for some general constructions, which parallels the discussion on Kalinin effectivity in the 
previous section.

\medskip

As a direct application of K\"unneth's formula and Proposition \ref{GM-char} we have:


\begin{prop}
\label{products-maximal}
A finite product of maximal (Galois maximal) pairs is a maximal (Galois maximal) pair.\qed
\end{prop}


\begin{prop}
\label{max-bundles}
Let $(X,\Conj)$ be a $G$-space, $\xi$ be a complex vector bundle on $X$ 
endowed with an anti-linear involution covering the involution $\Conj,$ and 
$Y$ an associated bundle of flags in $\xi$ equipped with the induced involution.
\begin{itemize}
\item[ 1)] $X$ is maximal if and only if $Y$ is maximal.
\item[ 2)] $X$ is Galois maximal if and only if $Y$ is Galois maximal.
\end{itemize}
\end{prop}


\proof
Let $\pi: Y\ra X$ denote the projection. From the Leray-Hirsch theorem 
for the fibrations $\pi:Y\ra X$ and $\pi_{|F(Y)}:F(Y)\ra F(X),$ we obtain
\begin{equation}
\label{l-h-max}
\beta_*(X)\beta_*(\Fl)=\beta_*(Y) \quad {\text{and}}\quad \beta_*(F(X))\beta_*(F(\Fl))=\beta_*(F(Y)).
\end{equation}
Since $\Fl$ is a maximal variety, we have $\beta_*(\Fl)=\beta_*(F(\Fl))$. 
Therefore, from (\ref{l-h-max}) we infer that the maximality of $X$ is 
equivalent to the maximality of $Y.$ Finally, the proof of item 2) follows 
immediately from the Leray-Hirsch isomorphism and Propositions 
\ref{multiplicative-kalinin} and \ref{GM-def}.
\qed


\subsubsection{Maximal blowing-up} 
\label{mbu}


As in Section \ref{ebu}, let $(Y,B, \Conj)$ be a $G$-pair, where $Y$ is a complex 
manifold equipped with an anti-holomorphic involution $\Conj$, and $B\subseteq Y$ is a 
$\Conj$-invariant complex submanifold of complex codimension at least two. Consider  
the blow-up  $\pi: X\ra Y$ of $Y$ along $B,$ and denote by $E$ the exceptional divisor.


\begin{prop}
\label{max-classic}
$X$ is maximal if and only if $Y$ and $B$ are maximal.
\end{prop}


\proof
Let $d=\codim_Y^\CC B.$ Due to the well-known isomorphisms 
\begin{align*}
\label{blowup-formulas}
H^*(X)\oplus H^*(B)=&H^*(Y)\oplus H^*(E),\notag \\
H^*(F(X))\oplus H^*(F(B))=&H^*(F(Y))\oplus H^*(F(E)),
\end{align*}
we find 
$$
\mathfrak D(X)+\mathfrak D(B)=\mathfrak D(Y)+\mathfrak D(E).
$$
On the other hand, by Leray-Hirsch theorem, we have
$$
\beta_*(E)=d\beta_*(B)\quad {\text{and}} \quad \beta_*(F(E))=d\beta_*(F(B)),
$$
and so 
\begin{equation}
\label{deficiency-blowup}
\mathfrak D(X)=\mathfrak D(Y)+(d-1)\mathfrak D(B).
\end{equation}
 Since $d\geq 2$ and the Smith-Thom deficiency is non-negative, the result 
 follows from (\ref{deficiency-blowup}).
\qed

\medskip

We conclude this section by recalling the following result of Derval 
\cite{derval} which establishes the behavior of Galois maximality 
under blow-up. For the reader's convenience, we sketch a short proof, 
different from Derval's, based on Proposition \ref{GM-char}.


\begin{prop}
\cite[Proposition 3.5]{derval}
\label{derval-GM}
If $Y$ and $B$ are Galois maximal, then $X$ is Galois maximal.
\end{prop}


\proof
By the blow-up formula, for every $i\geq 0,$ there exists a $G$-invariant isomorphism 
$$
H^i (Y)\oplus \bigoplus_{k\geq 1} H^{i-2k}(B)\to   H^i(X),
$$
given by  
$$
(y,b_1,b_2,\dots)=\pi^*(y)+\sum_{k\geq 1} j_*(\pi_E^*(b_k)\xi^k )
$$ 
where $\pi_{E}:E\to B$ is the restriction of $\pi$, $E$ is identified with $\PP(N_{B|Y})$,
and $\xi$ stands for $c_1(\OO_E(1))\pmod 2\in H^2(E),$ (see, for example, 
\cite[Th. 7.31]{voisin} and its proof). Thus, applying Proposition \ref{GM-char} 
to the Poincar\'e dual isomorphisms,  we obtain the result.
\qed

\begin{rmk}
{\rm In general, the examples found in Propositions \ref{hyperquadrics}, \ref{effective-GM}, 
\ref{sym-kalinin}, and \ref{blowup-emptyset} while effective, are not conjugation spaces.
This shows that the class of effective spaces is strictly larger than the class of conjugation spaces.}
\end{rmk}


\section{Kalinin effectivity and stretchedness}
\label{stretchedG-pairs}


Let $(Y,B,\Conj)$ be a $G$-pair, where $Y$ is a complex manifold equipped 
with an anti-holomorphic involution $\Conj$, and $B\subseteq Y$ is a 
$\Conj$-invariant complex submanifold of complex codimension at least two. 
As suggested by Proposition \ref{effectivity-blowup}, stability of Kalinin effectivity 
under blowing-up requires both the effectivity of the $G$-space $B$ and that 
of the $G$-pair $(Y,B)$. For stability of Smith-Thom maximality and Galois 
maximality under blow-up, one typically assumes Smith-Thom maximality 
(respectively Galois maximality) of both $B$ and $Y.$ In this section we introduce 
a class of $G$-pairs $(Y,B)$ for which Kalinin effectivity or maximality of $B$ 
and $(Y,B)$ is equivalent to Kalinin effectivity or maximality of $B$ and $Y.$ 
This will be useful for the applications in the next section.

Before doing so, we consider the case in which both Kalinin effectivity 
and maximality are assumed for $B$ and $(Y,B)$.


\begin{prop}
\label{conj_spaces-blowup}
If $B$ and $(Y,B)$ are maximal and effective, then $Y$ is a conjugation space.
\end{prop}


\proof
From the exact sequence of the pair $(Y,B)$
$$
\cdots \xrightarrow{\delta^{i-1}} H^i(Y,B) \ra H^i(Y) \ra H^i(B)\xrightarrow{\delta^i}\cdots,
$$
we find that the total Betti numbers satisfy
\begin{equation}
\label{sum-betti-cx}
\beta_*(Y)=\beta_*(B)+\beta_*(Y,B)-2\sum_{i\geq 0} \dim \im\delta^i.
\end{equation}
Since $B$ and the $G$-pair $(Y,B)$ are effective, by Proposition 
\ref{delta-vanish}, for every $i\geq 0$ the boundary operators 
$\delta_\RR^i:H^i(F(B))\ra H^{i+1}(F(Y),F(B))$ 
vanish, and so we have short exact sequences:
\begin{equation}
\label{eff-real-stretch}
0\ra H^i(F(Y),F(B))\ra \, H^i(F(Y))\ra H^i(F(B))\ra 0,
\end{equation}
and so the total Betti numbers satisfy
\begin{equation}
\label{sum-betti-real}
\beta_*(F(Y))=\beta_*(F(B))+\beta_*(F(Y),F(B)).
\end{equation}
Since $B$ and the $G$-pair $(Y,B)$ are maximal, we have 
$$
\beta_*(F(B))=\beta_*(B) \quad{\text{and}}\quad \beta_*(F(Y),F(B))=\beta_*(Y,B),
$$ 
while Smith inequality gives
$\beta_*(F(Y))\leq \beta_*(Y).$
From  (\ref{sum-betti-cx}) and  (\ref{sum-betti-real}) we obtain
\begin{align*} 
\beta_*(B)+\beta_*(Y,B)-2\sum_{i\geq 0}  \dim \im\delta^i
=&\, \beta_*(Y)\\
\geq  &\, \beta_*(F(Y))\\
=& \, \beta_*(F(B))+\beta_*(F(Y),F(B))\\
= &\, \beta_*(B)+\beta_*(Y,B).
\end{align*}
Therefore, $Y$ is maximal and the boundary maps $\delta^i$ are trivial 
for all $i\geq 0.$ To prove that $Y$ is effective, it remains to observe
that from the exact sequence (\ref{eff-real-stretch}), we obtain 
an induced short exact sequence on filtrations:
$$
0\to \ff_i(Y,B)\to \ff_i(Y)\to \ff_i(B)\to 0,
$$
and so 
$$
0\to \Sq^{-1}\ff_i(Y,B)\to \Sq^{-1}\ff_i(Y)
\to  \Sq^{-1}\ff_i(B)\to 0.
$$
From the effectivity of  $(Y,B)$ and $B$, we have
$$
\Sq^{-1}\ff_i(Y,B)=H^{\le \frac{i}2}(F(Y),F(B)) \quad{\text{and}}\quad 
\Sq^{-1}\ff_i(B)=H^{\le \frac{i}2}(F(B)).
$$
Since the inclusion and relative homomorphisms respect the dimension grading, 
we find that
$$
\Sq^{-1}\ff_i(Y)=H^{\le \frac{i}2}(F(Y)),
$$
and so $Y$ is effective. Thus, by Proposition \ref{equivalence}, 
being both effective and maximal, $Y$ is a conjugation space.
\qed


\begin{cor}
\label{induction-max+eff}
If $B$ and $(Y,B)$ are maximal and effective, $X$ is the blow-up of $Y$ along $B,$ 
and $E$ is the exceptional divisor, then $X,\, E$ and $(X,E)$ 
are maximal and effective.
\end{cor}


\proof
The result follows from Proposition \ref{conj_spaces-blowup}, and 
Propositions \ref{effective-bundles}, \ref{effectivity-blowup} and 
\ref{max-classic}.
\qed


\begin{rmk}
{\rm 
The proof of Proposition \ref{conj_spaces-blowup} suggests that 
Kalinin effectivity and maximality of $G$-pairs have the same good 
behavior when the boundary maps $\delta^i$ and $\delta^i_\RR$ 
vanish for every $i\geq 0.$
}

\end{rmk}


\subsection{Stretched $G$-pairs} 
\label{stretchedness}


Let $(X,\Conj)$ be a connected compact complex manifold equipped with 
an anti-holomorphic involution $\Conj.$ 
\begin{defn}
\label{pair-stretch}
A $G$-pair $(A,B)$ of $\Conj$-invariant subspaces of $X$ is called a 
{\em stretched $G$-pair} if for every $p\geq 0,$ the inclusion homomorphisms
\begin{equation}
\label{str-cond-pair}
\begin{aligned}
H^p(A)\to &\, H^p(B), \\
H^p(F(A))\to &\, H^p(F(B))
\end{aligned} 
\end{equation}
are epimorphisms. 
\end{defn}


\begin{lem}
\label{triple-stretch}
Let $(A,B)$ and $(B,C)$ be two $G$-pairs  of $\Conj$-invariant 
subspaces of $X.$ 
\begin{itemize}
\item[ 1)] If the $G$-pair $(A,C)$ is stretched, then the $G$-pair $(B,C)$ 
is stretched.
\item[ 2)] If the $G$-pairs $(A,B)$ and $(B,C)$ are stretched, then the 
$G$-pair $(A,C)$ is stretched.
\end{itemize}
\end{lem}


\proof
Let $i: B\to A,\,j: C\to B$ and $k:C\to A$ denote the inclusions, and 
$F(i): F(B)\to F(A),$ $F(j): F(C)\to F(B)$ and $F(k):F(C)\to F(A),$ are the 
corresponding inclusions of fixed loci. By functoriality, for every $p\geq 0$ 
we have the following commutative diagrams:
\begin{figure}[h]
\hspace*{-1in}    
\centering
    \begin{minipage}{0.4\textwidth}
        \centering
        \leavevmode 
       \xymatrix{ 
&H^p(A)\ar[rd]_{k^*}\ar[r]^{i^*}  & H^p(B)  \ar[d]^{j^*}\\
& &  H^p(C)
 } 
    \end{minipage}
    \begin{minipage}{0.4\textwidth}
        \centering
        \leavevmode 
       \xymatrix{ 
&H^p(F(A))\ar[rd]_{F(k)^*}\ar[r]^{F(i)^*}  & H^p(F(B))  \ar[d]^{F(j)^*}\\
& &  H^p(F(C)).
 }        
    \end{minipage}
    \caption{Stretchedness for triples.}
\end{figure}

Since $j^*\circ i^*=k^*$ and $F(j)^*\circ F(i)^*=F(k)^*,$ the conclusions 
of the lemma follow.
\qed


\begin{prop}
\label{stretch-effective}
Let $(Y,B)$ be a stretched $G$-pair. Then 
\begin{itemize}
\item[ 1)] If $Y$ and $B$ are effective then $(Y,B)$ is effective.
\item[ 2)] If $B$ and $(Y,B)$ are effective then $Y$ is effective.
\item[ 3)] If $Y$ and $(Y,B)$ are effective then $B$ is effective.
\end{itemize}
\end{prop}


\proof
Since the $G$-pair $(Y,B)$ is stretched, for every integer $i\geq 0$ 
we have a short exact sequence
\begin{equation}
\label{stretch-real}
0\ra H^i(F(Y),F(B))\ra H^i(F(Y))\ra H^i(F(B))\ra 0.
\end{equation}
By applying Kalinin's spectral sequence, for each $i\geq 0$ 
we get a short exact sequence 
$$
0\to \ff_i(Y,B)\to \ff_i(Y)\to \ff_i(B)\to 0.
$$
Therefore, we have a short exact sequence
\begin{equation}
\label{inverse-sq-ses}
0\to \Sq^{-1}\ff_i(Y,B)\to \Sq^{-1}\ff_i(Y)\to  \Sq^{-1}\ff_i(B)\to 0.
\end{equation}
Since the inclusion and relative homomorphisms respect the dimension 
grading, from the short exact sequences 
(\ref{stretch-real}) and 
(\ref{inverse-sq-ses}) we infer that if two of the following relations 
\begin{align*}
\Sq^{-1}\ff_i(Y,B)=&\, H^{\leq i/2}(F(Y,B)),\\
\,\Sq^{-1}\ff_i(Y)=&\, H^{\leq i/2}(F(Y)),\\
\Sq^{-1}\ff_i(B)=&\, H^{\leq i/2}(F(B))
\end{align*}
hold, then so does the third.
\qed

\medskip

An analogous result holds for Smith-Thom maximality:


\begin{prop}
\label{stretch-max}
Let $(Y,B)$ be a stretched $G$-pair. Then 
\begin{itemize}
\item[ 1)] If $Y$ and $B$ are maximal then $(Y,B)$ is maximal.
\item[ 2)] If $B$ and $(Y,B)$ are maximal then $Y$ is maximal.
\item[ 3)] If $Y$ and $(Y,B)$ are maximal then $B$ is maximal.
\end{itemize}
\end{prop}


\proof
By assumption, for every $i\geq 0$ we have short exact sequences
\begin{equation}
\label{pairwise-stretching}
\begin{aligned}
0\ra H^i(Y,B)\ra &\, H^i(Y) \ra H^i(B)\ra 0\\
0\ra H^i(F(Y),F(B))\ra &\, H^i(F(Y))\ra H^i(F(B))\ra 0.
\end{aligned}
\end{equation}
As a consequence, we find 
$$
\beta_*(Y)=\beta_*(B)+\beta_*(Y,B)\,\,{\text{and}}\,\, \beta_*(F(Y))=\beta_*(F(B))+\beta_*(F(Y,B)),
$$
and the claims follow immediately.
\qed

\smallskip

The following examples show that no direct analogue of Propositions 
\ref{stretch-effective} and \ref{stretch-max} holds for Galois maximality.


\begin{examples}
\label{ex1}
{\rm 
Let $(Y, \Conj)$ be the {\it quadric ellipsoid,} i.e. $Y = \PP^1\times \PP^1$ 
equipped with an anti-holomorphic involution whose fixed locus is 
homeomorphic to $S^2$ and let $B$ be a hyperplane section with no 
real points. In this case, we can immediately see that the pair $(Y,B)$ is 
stretched. By Proposition \ref{effective-GM}, $Y$ is Galois maximal. 
In addition, the Galois maximality of $(Y,B)$ is equivalent to the Galois 
maximality of $Y\setminus B.$ The latter admits a $G$-equivariant 
deformation retractionto $(F(Y), \Conj=\id)$ which is Galois maximal 
since $\Conj$ acts as the identity. Hence, $(Y,B)$ is Galois maximal. 
However, $B$ is not Galois maximal since $F(B)=\emptyset.$
}
\end{examples}


\begin{examples}
\label{ex2}
{\rm 
Let $(B,\Conj_B)$ be the quadric ellipsoid and consider the interval 
$[-1,1]$ equipped with the involution  $\Conj_{[-1,1]}(t)=-t.$ Let $Y$ 
be obtained from $B$  and $[-1,1]$ by identifying  the pair $\{\pm 1\}$ 
with a pair of complex conjugate points on a purely imaginary hyperplane 
section $H$ of $B$. We  equip $Y$ with the involution $\Conj$ induced 
from $\Conj_B$ and $\Conj_{[-1,1]}.$ Then $(Y,B)$ is a Galois maximal 
$G$-pair, since the involution $\Conj_{[-1,1]}$ restricted to 
$Y\setminus B=(-1,1)$ is maximal. Notice now that while $B$ is also 
Galois maximal, $Y$ is not Galois maximal. Indeed, if $\xi\in H_1(Y)$ 
is the $1$-cycle given by the interval $[-1,1]$ completed by a meridian 
in $H,$ then the Kalinin differential on the second page is
$$
d^2(\xi)=[H]\neq 0\in \prescript{2}{}H_2=H^1(G, H_2(Y))=\FF_2[H].
$$ 
It should also be noted that in this example the $G$-pair $(Y,B)$ is 
stretched.
}
\end{examples}


\subsection{On the absence of stretchedness}


In general, in the absence of stretchedness, the behavior of Kalinin effectivity 
and maximality exhibited in Propositions \ref{stretch-effective} and \ref{stretch-max}, 
respectively, fails, as the following examples show. These examples show that 
stretchedness is not merely technical, but is genuinely needed in the cases 
discussed in Sections \ref{Wond-Compact} and \ref{examples}.


\begin{examples}
{\rm 
Let $Y=\CC\PP^2$ be equipped with the standard conjugation, and let 
$B\subseteq \CC\PP^2$ be a nonempty real conic. Then, the $G$-pair $(Y,B)$ 
is not stretched, while both $Y$ and $B$ are effective. However, since 
$F(Y\setminus B)$ is not connected, we see that the $G$-pair $(Y,B)$ 
is not effective.
}
\end{examples}


\begin{examples}
{\rm 
Let $(Y,\Conj)$ be the quadric ellipsoid, and $B$ a nonsingular hyperplane 
section with real points. Then, $B$ and the pair $(Y,B)$ are maximal, while 
$Y$ is not. Notice that in this case the map $H^1(F(Y))\to H^1(F(B))$ is 
the zero map, and so the pair $(Y,B)$ is not stretched.
}
\end{examples}


\subsection{Constructions preserving stretchedness}


We conclude this section with two results which discuss the behavior of stretchedness in two 
cases relevant for the applications discussed in the next section.

\begin{prop}
\label{proj-stretch}
Let $(Y,\Conj)$ be a compact complex manifold equipped with 
an anti-holomorphic involution and $Z\subseteq Y$ a $\Conj$-invariant submanifold. 
If $\EE$ is a holomorphic vector bundle $Y$
equipped with an anti-linear involution $\hat \Conj$ covering the involution $\Conj,$ and
$\ff$ a $\hat \Conj$-invariant subbundle of $\EE,$ then:
\begin{itemize} 
\item[ 1)] 
The  $G$-pair $(\PP(\EE),\PP(\ff))$ is stretched.
\item[ 2)] If the $G$-pair $(Y,Z)$ is stretched, then the $G$-pair 
$(\PP(\EE),\PP(\EE)_{|Z})$ is stretched.
\end{itemize}
\end{prop}

\proof
Both items are direct consequences of the Leray-Hirsch theorem, since 
the tautological bundles of $\PP(\ff)$ and $\PP(\EE_{|Z})$ are restrictions 
of the tautological bundle of $\PP(\EE).$ In particular, the restriction 
maps on cohomology are surjective in each degree, both on the complex 
and real loci, so the pairs are stretched by Definition \ref{pair-stretch}.
\qed

\begin{prop}
\label{exceptional-stretch}
Let $(Y,\Conj)$ be a compact connected complex manifold equipped with 
an anti-holomorphic involution and  $Z$ a $\Conj$-invariant submanifold of $Y.$ 
If the $G$-pair $(Y,Z)$ is stretched, then the $G$-pair $(\wt Y, E)$ is 
stretched, where $\wt Y$ is the blow-up of $Y$ along $Z$ and $E$ is the 
exceptional divisor of the blow-up.
\end{prop}

\proof Let $\pi:\wt Y\to Y$ be blow-up map, and denote by $i: Z\to Y$ and 
$j:E\to \wt Y$ the inclusion maps. 
By the Leray-Hirsch theorem, for every $p\geq 0$ and 
$x\in H^p(E)=H^p(\PP(N_{Z|Y})),$ we have  
$$
x=\sum_{i+2k=p}(\pi^* a_i)\xi^k,
$$
where $\xi=c_1(\OO_E(1))\pmod 2\in H^2(E)$ and $a_i\in H^i(Z).$ 
Notice now that
$c_1(\OO_E(1))=j^*\eta,$ where $\eta=c_1(\OO_{\wt Y}(-E))\in H^2(\wt Y).$ 
Moreover, since $(Y, Z)$ is stretched,  for every $i\geq 0$ there exists  
$b_i\in H^i(Y)$ such that $b_i=i^*(a_i).$ Therefore
$$
x=\sum_{i+2k=p}\pi^* a_i\xi^k=\sum_{i+2k=p}(\pi^* i^*b_i) (j^*\eta)^k=
j^*(\sum_{i+2k=p}(\pi^*b_i)\eta^k)\in H^p(\wt Y),
$$
and so the inclusion homomorphism $j^*:H^p(\wt Y)\to H^p(E)$ is surjective.

To prove the surjectivity of the inclusion 
homomorphism $H^p(F(\wt Y))\to H^p(F(E)),$ 
we proceed 
in the same manner. By using again the Leray-Hirsch theorem, 
for every $p\geq 0$ and 
$$
x\in H^p(F(E))=H^p(F(\PP(N_{Z|Y})))=H^p(\PP_\RR(N_{F(Z)|F(Y)})),
$$ 
 we find
$$
x=\sum_{i+k=p}(\pi^* a_i)\xi^k,
$$
where $a_i\in H^i(F(Z))$ and $\xi=w_1(\OO_{F(E)}(1))\in H^1(F(E)).$  
We notice next  that
$w_1(\OO_E(1))=j^*\eta,$ where $\eta=w_1(\OO_{\wt Y}(-E))\in H^1(F(\wt Y)).$ 
Moreover, since $(Y, Z)$ is stretched, for every $i\geq 0$ there exists 
$b_i\in H^i(F(Y))$ such that $b_i=i^*(a_i).$  Therefore
$$
x=\sum_{i+k=p}\pi^* a_i\xi^k=\sum_{i+k=p}(\pi^* i^*b_i) (j^*\eta)^k=
j^*(\sum_{i+k=p}(\pi^*b_i)\eta^k)\in H^p(F(\wt Y)).
$$
Hence, the map $j^*:H^p(F(\wt Y))\to H^p(F(E))$ is surjective.
\qed


\section{Wonderful Compactifications}
\label{Wond-Compact}


Let $X$ be a compact connected complex manifold.  An {\it arrangement} 
$\aa$ of submanifolds in $X$ is  a finite collection of submanifolds in $X,$ 
not necessarily connected, such that for every $A, B\in \aa$ we have 
$A\cap B\in \aa$ and their tangent bundles satisfy $T(A\cap B)=T A\cap T B$, 
i.e., $A$ and $B$ intersect cleanly.

\begin{defn}
Let $(X, \Conj)$ be a compact connected complex manifold equipped with 
an anti-holomorphic involution. An arrangement $\aa$ of submanifolds in 
$X$ is called a {\em $G$-arrangement} if $\Conj(A)=A$ for every 
$A\in \aa$.
\end{defn}

We introduce next a $G$-equivariant analogue of  Li's construction 
of  wonderful compactifications of arrangements. The following definitions 
are adaptations of \cite[Definition 2.2]{li} to our $G$-equivariant setting.
\begin{defn}
Let $(X, \Conj)$ be a compact connected complex manifold equipped 
with an anti-holomorphic involution. A subset $\BB$ of a $G$-arrangement 
$\aa$   of submanifolds in $X$ is called a {\em $G$-building set of $\aa$} 
if for every $A\in \aa$, the minimal elements of$\{B\in\BB : B\supseteq A\}$ 
intersect transversally and their intersection is $A.$
\end{defn}

\begin{defn}
Let $(X, \Conj)$ be a compact connected complex manifold equipped with an 
anti-holomorphic involution. A finite set $\BB$ of  $\Conj$-invariant submanifolds 
of $X$ is called a {\em $G$-building set} if the set of all possible intersections of 
collections of submanifolds from $\BB$ forms a $G$-arrangement $\aa$ and 
that $\BB$ is a $G$-building set of $\aa.$ In this case, $\aa$ is called the 
{\em arrangement induced by $\BB.$}
\end{defn}

For any $G$-arrangement $\aa$ of submanifolds in $X,$ we define
$$
X^\circ:=X\setminus\bigcup_{A\in\aa}A.
$$

If $\BB$ is a $G$-building set of a $G$-arrangement $\aa,$ we have
$$
\bigcup_{B\in \BB} B= \bigcup_{A\in \aa} A\quad{\text{and}} \quad
\bigcup_{B\in \BB} F(B)= \bigcup_{A\in \aa} F(A),
$$ 
and so
$$X^\circ=X\setminus\bigcup_{B\in\BB}B \quad{\text{and}} \quad
F(X^\circ)=F(X)\setminus\bigcup_{B\in\BB}F(B).
$$

\medskip

We define ({\it cf.} \cite[Definition 1.1]{li}) the wonderful compactification 
of $X^\circ$ induced by a $G$-building set $\BB$ as follows:

\begin{defn}
\label{def-wonderful-compactification} 
Let $(X, \Conj)$ be a compact connected complex manifold equipped 
with an anti-holomorphic involution,  and $\BB$ a $G$-building set of 
submanifolds in $X.$ The closure of the image of the diagonal embedding
$$
X^\circ\hookrightarrow \prod_{B\in\BB}\Bl_B X,
$$
is called the {\em wonderful compactification} of the arrangement 
induced by $\BB,$ and is denoted by $X_\BB$.
\end{defn}

\begin{defn}
Let $X$ be a compact connected complex manifold and $\pi: Bl_C X\to X$ 
the blow-up of $X$ along  a  submanifold $C.$ For any analytic subspace 
$V \subseteq X,$ the dominant transform of $V,$ denoted by $\wt V,$ 
is defined as  the proper transform of $V$ if $V\not\subseteq C,$ and 
$\pi^{-1}(V)$ if $V\subseteq C.$ 
\end{defn}


\begin{prop}
\label{blowup-arrangement}
Let $(X,\Conj)$ be a compact connected complex manifold equipped with 
an anti-holomorphic involution, $\BB$ a $G$-building set with the induced 
$G$-arrangement $\aa$ of submanifolds in $X.$ Let $F$ be a minimal 
element in $\BB$ and let $\pi : \Bl_F X\ra X$ be the blow-up of $X$ along 
$F.$ Denote the exceptional divisor by $E.$
\begin{itemize}
\item[ 1)] The collection $\aa_F$ of submanifolds in $\Bl_F X$ defined as 
$$
\aa_F := \{\wt A\}_{A\in \aa} \cup \{\wt A \cap E\}_{A\in \aa,
\, \emptyset\subsetneq A\cap F\subsetneq A}
$$
is a $G$-arrangement of submanifolds in $\Bl_F X.$
\item[ 2)] $\displaystyle \BB_F:= \{\wt B\}_{B\in \BB}$ is a building set of $\aa_F.$
\end{itemize}
\end{prop}


\proof
The proof follows from \cite[Proposition 2.8]{li}, noting that since $X$ is 
equipped with an anti-holomorphic involution $\Conj$ and $F$ is a 
$\Conj$-invariant submanifold of $X,$ then the blow-up $\Bl_F X$  
admits a naturally induced anti-holomorphic involution $\hat\Conj$ 
and the dominant transform of a $\Conj$-invariant submanifold is 
$\hat\Conj$-invariant. 
\qed

\medskip

We call $\aa_F$ and $\BB_F$ the {\it blow-up $G$-arrangement} and 
its {\it blow-up $G$-building set}, respectively.

\medskip

Let $(X, \Conj)$ be a compact connected complex manifold equipped 
with an anti-holomorphic involution,  $\aa$ a $G$-arrangement of 
submanifolds in $X,$ and $\BB=\{B_1,\dots, B_N\}$  a $G$-building 
set of $\aa.$ We now describe the standard iterative construction of 
the wonderful compactification via successive blow-ups along the 
elements of the building set.

\medskip

For every $k=0,\dots,N,$ we inductively define a triple 
$$
(X_k, \aa^{(k)},\BB^{(k)})
$$ 
consisting of a compact connected complex manifold $(X_k,\Conj_k)$ 
equipped with an anti-holomorphic involution,  a $G$-arrangement 
$\aa^{(k)}$ of submanifolds in $X_k,$ and  a $G$-building set $\BB^{(k)}$, 
as follows:

\medskip

(i) For $k=0$, define $X_0=X$, $\aa^{(0)}=\aa$,
$\BB^{(0)}=\BB$. 

(ii) Assume that $(X_{k-1},\aa^{(k-1)},\BB^{(k-1)})$ is constructed. 
\begin{itemize}
\item Define $X_{k}$ to be the blow-up of $X_{k-1}$ along the 
submanifold $B^{(k-1)}_k,$  and denote by $E^{(k)}$ the exceptional divisor.

\item 
For each $B\in\BB^{(k-1)}$, let $B^{(k)}$ denote its dominant transform in 
$X_k$, and define 
$$
\BB^{(k)}:=\{B^{(k)}\}_{B\in\BB^{(k-1)}}.
$$

\item Define the collection of submanifolds of $X_k$ given by 
$$ 
\aa^{(k)}:=\{\wt{A}\}_{A\in\aa^{(k-1)}} \cup\{\wt{A}
\cap E^{(k)}\}_{\emptyset\subsetneq A\cap B_k^{(k-1)}\subsetneq A}.
$$ 
Since $B^{(k-1)}_k$ is assumed $\Conj_{k-1}$-invariant, by Proposition 
\ref{blowup-arrangement} and \cite[Proposition 2.8]{li},  $\aa^{(k)}$ is a 
$G$-arrangement of submanifolds in $X_k$ equipped with the induced 
anti-holomorphic involution $\Conj_k$ and $\BB^{(k)}$ is a $G$-building 
set of $\aa^{(k)}.$
\end{itemize}

(iii) Continue the inductive construction until $k=N$. We obtain $$(X_N, \aa^{(N)},
\BB^{(N)})$$ where all the subvarieties in the building set
$\BB^{(N)}$ are divisors.

\smallskip

We define
$$
\Bl_\BB X:=X_N,\quad \aa_\BB:=\aa^{(N)},\quad \wt {\BB}:=\BB^{(N)}.
$$
By construction, $\Bl_\BB X$ is a compact complex manifold equipped with 
an anti-holomorphic involution, and $\aa_\BB$ is a $G$-arrangement of 
submanifolds in $\Bl_\BB X$ induced by the $G$-building set $\wt \BB.$ 

\smallskip

The following theorem summarizes the basic properties of the 
wonderful compactification, and is a direct consequence of  
\cite[Proposition 2.13, Theorem 1.2 and Theorem 1.3]{li}:


\begin{thm}
\label{main_wonder}
Let $(X,\Conj)$  be a compact complex manifold equipped with 
an anti-holomorphic involution and a nonempty $G$-building
set $\BB=\{B_1,\dots, B_N\},$ and let 
$$
\displaystyle X^{\circ}=X\setminus \bigcup_{i=1}^N B_i.
$$ 
Then:
 \begin{itemize}
 \item[ 1)]  
 $\Bl_\BB X$ is the wonderful compactification of the $G$-arrangement 
 induced by the $G$-building set $\BB.$

\item[ 2)] For each $B_i\in \BB$
 there is a nonsingular divisor $D_{B_i}\subset \Bl_\BB X$, such that
$\displaystyle D_\BB=\bigcup_{i=1}^N D_{B_i}=\Bl_\BB X\setminus  X^\circ$ 
and $D_\BB$ is a normal crossing divisor.

\item[ 3)] Let $\I_i$ be the ideal sheaf of $B_i\in\BB$.  The wonderful 
compactification $\Bl_\BB X$ is isomorphic to the blow-up of \,$X$ 
along the ideal sheaf $\I_1\I_2\cdots \I_N$.

\item[ 4)] 
If we arrange $\BB=\{B_1,\dots, B_N\}$ in such an order that the first $i$ terms 
$B_1,\dots, B_i$ form a $G$-building set for any $1\le i\le N,$ then
\begin{equation}
\label{iterated-bu}
\Bl_\BB X=\Bl_{\widetilde{B}_N}\cdots \Bl_{\widetilde{B}_2}\Bl_{B_1} X,
\end{equation} 
where each blow-up is along a nonsingular subvariety. \qed
\end{itemize}
\end{thm}


\subsection{Stretched $G$-arrangements}
\label{arrangements}


\begin{defn}
Let $(X, \Conj)$ be a compact complex manifold equipped with an 
anti-holomorphic involution, and let $\aa$ be a $G$-arrangement 
of submanifolds in $X.$
\begin{itemize}
\item[ 1)]
The $G$-arrangement $\aa$ is called {\em stretched} if for any  $A\in \aa$ the 
inclusion homomorphisms 

\begin{equation}
\label{str-cond}
\begin{aligned}
H^p(X)\to &\, H^p(A), \\
H^p(F(X))\to &\, H^p(F(A))
\end{aligned} 
\end{equation}
are surjective. 

\item[ 2)] The $G$-arrangement $\aa$ is called effective, maximal 
or Galois maximal, if $X$ and every $A\in \aa$  are effective, maximal 
or Galois maximal, respectively.
\end{itemize}
\end{defn}

\begin{rmk}
{\rm 
If a $G$-arrangement $\aa$ is stretched, then by Lemma \ref{triple-stretch} 
one can see that for any $A, B\in \aa$ such that $B\subseteq A,$ the 
$G$-pair $(A,B)$ is stretched.}
\end{rmk}


\begin{prop}
\label{induction-stretch}
Let $(X, \Conj)$ be a compact complex manifold equipped with an anti-holomorphic 
involution, $\aa$ a  $G$-arrangement of submanifolds in $X,$ $\BB$ a 
$G$-building set of $\aa$, and let $B\in \BB$ be a minimal element 
with respect to inclusion.  
Consider the blow-up $G$-arrangement $\aa_B$ in the blow-up 
$\pi: \Bl _B X\to X.$ Then
\begin{itemize}
\item[ 1)] 
If the $G$-arrangement $\aa$ is stretched, then the blow-up 
$G$-arrangement $\aa_B$ is stretched. 
\item[ 2)] 
If the $G$-arrangement $\aa$ is  stretched and effective, then 
the blow-up $G$-arrangement $\aa_B$ is effective.
\item[ 3)] If the $G$-arrangement $\aa$ is maximal or Galois maximal, 
then the blow-up $G$-arrangement $\aa_B$ is maximal or Galois maximal, 
respectively. 
\end{itemize}
\end{prop}


\proof 
Let $i:A\to X$ and $j:B\to X$ denote the inclusion maps. We denote by 
$E=\PP(N_{B|X})$ the exceptional divisor of the blow-up, and by 
 ${\wt A}$ the dominant transform of $A.$ Let  $\wt j: E\to  \Bl_B X$ 
 and $\wt i:\wt A\to  \Bl_B X$ be the inclusion  maps, $\pi_A, f_B$ and  
 $f_A$ the restriction of $\pi$ to ${\wt A}$ and $E,$ respectively.
These maps give rise to the following commutative diagram:

\begin{equation}
\label{standard-diag}
\xymatrix{ 
&\wt A\ar[d]_{\pi_A}\ar[r]^{\wt i}  &  \Bl_B X  \ar[d]^{\pi} & E\ar[l]_{\wt j}\ar[d]^{f_B}\\
&A \ar[r]_i&  X&B\ar[l]^j.
 }
\end{equation}

\smallskip

Recall that the blow-up arrangement $\aa_B$ on $ \Bl_B X$ is defined by
$$
\aa_B=\{\wt{A}\}_{A\in\aa} \cup
\{\wt{A}\cap E\}_{\emptyset\subsetneq A\cap B\subsetneq A}.
$$

\medskip

Before we proceed, notice that if $B$ and $X$ are effective, maximal or 
Galois maximal, then by Propositions \ref{stretch-effective},  
\ref{effectivity-blowup}, \ref{max-classic}, \ref{derval-GM}, we find that 
$\Bl_B X$ is effective, maximal or Galois maximal, respectively. 
We prove the three statements simultaneously. Since stretchedness 
concerns the surjectivity of restriction maps, we examine 
each submanifold appearing in the blow-up arrangement $\aa_B$. 
This leads to four cases depending on the relative position of $A\in\aa$ with 
respect to the blow-up center $B$.

\medskip

{\em Case 1:} $A\cap B=\emptyset.$

\smallskip

Since $A\cap B=\emptyset,$ the map $\pi_A^*$ is the identity, and so  
$i^*={\wt i}^*\circ\pi^*.$ The stretchedness of the $G$-pair $(X,A)$ implies 
that the map $i^*$ is surjective, and so ${\wt i}^*$ is surjective, as well. 
Similarly, we find that $H^*(F( \Bl_B X))\to H^*(F(\wt A))$ is surjective, 
which shows that the $G$-pair $( \Bl_B X, \wt A)$ is stretched. 

It remains to notice that, since the map $\pi_A^*$ is the identity, if $A$ 
is effective, maximal or Galois maximal, then $\wt A$ satisfies the 
corresponding property as well.

\medskip

{\em Case 2:} $A\subseteq  B.$

\smallskip

In this case $\wt A=\pi^{-1}(A)$ coincides with the projectivized normal 
bundle $\PP(N_{B|X}|_A)\subseteq E$. By Lemmas \ref{triple-stretch} and 
\ref{exceptional-stretch}, it suffices to show that the $G$-pair $(E,\wt A)$ 
is stretched. However, since $E=\PP(N_{B|X}),$ this follows from 
Lemma \ref{proj-stretch}.

If $A$ is effective, from Proposition \ref{effective-bundles} we conclude 
that$\wt A=\PP({N_{B|X}}_{|A})\subseteq E$ and, consequently, 
$\wt A\cap E=\wt A$ are effective. Similarly, if $A$ is maximal or Galois 
maximal, then by Proposition \ref{max-bundles}, $\wt A$ and consequently, 
$\wt A\cap E=\wt A$ are maximal or Galois maximal, as well.

\medskip

{\em Case 3:} $B\subseteq  A.$

\smallskip

Let $\ell: B\to A$ denote the inclusion. In this case, $\wt A$ is the blow-up of $A$ 
along $B$ and  ${\wt A} \cap E$ is the exceptional divisor $E_A=\PP(N_{B|A}).$ 
Notice we have a commutative diagram:

\begin{equation}
\label{diagram-BinA}
\xymatrix{ 
&E_A\ar[d]_{f_A}\ar[r]^{\wt \ell}  & \wt A \ar[r]^{\wt i} \ar[d]^{\pi_A} &  \Bl_B X\ar[d]^{\pi}\\
&B \ar[r]_\ell&  A\ar[r]_i&X,
 }
\end{equation}
where  ${\wt \ell}:E_A\to {\wt A}$ denotes the inclusion, and $f_A:E_A\to B$ is the 
restriction of $\pi$ to $E_A.$

We prove first that the $G$-pair $( \Bl_B X, \wt A\cap E)$ is stretched. 
Notice that $N_{B|A}$ is a holomorphic sub-bundle of $N_{B|X},$ and 
since $A$ and $B$ are $\Conj$-invariant, it carries anti-linear involution 
covering the restriction of $\Conj$ to $B.$ Therefore, by Lemma \ref{proj-stretch}, 
the $G$-pair $(E, E_A)$ is stretched. Since $\wt A\cap E=E_A,$ using Lemmas 
\ref{triple-stretch} and \ref{exceptional-stretch}, we can now see that $G$-pair 
$( \Bl_B X, \wt A\cap E)$
is stretched.

To prove that the $G$-pair $( \Bl_B X, \wt A)$  is stretched, 
let $R$ be the polynomial ring $\FF_2[v],$ where $v$ is of degree $2,$ 
and consider the diagram
\begin{equation}
\label{blowup-formula-BinA}
\xymatrix{ 
&H^*(X)\otimes R \ar[d]_{i^* \otimes Id}\ar[r]^{\quad\phi_X}  & H^*( \Bl_B X)\ar[d]^{{\wt i}^*}\ \\
& H^*(A)\otimes R  \ar[r]_{\quad\phi_A}& H^*(\wt A).}
\end{equation}
 The horizontal maps are given by
 \begin{equation}
 \label{gitler-map-X}
 \phi_X(x\otimes v^r)=
 \begin{cases}
 \pi^*x,\, {\text{if}}\, r=0\\
{\wt j}_*(f_B^*j^*x\cup\xi^r),\, {\text{if}}\, r>0
\end{cases}
\end{equation}
and
\begin{equation}
\label{gitler-map-A-B}
\phi_A(a\otimes v^r)=
\begin{cases}
\pi_A^*a,\, {\text{if}}\, r=0\\
{\wt \ell}_*(f_A^*\ell^*a\cup\xi_A^r),\, {\text{if}}\, r>0,
\end{cases}
 \end{equation}
where $\xi$ and $\xi_A$ stand for the reduction mod $2$ of the first Chern 
class of the tautological line bundles of $E$ and $E_A,$ respectively.

\smallskip

Let us first show 
that the diagram (\ref{blowup-formula-BinA}) is commutative. 
Indeed, when $r=0,$
due to the commutativity of diagram (\ref{diagram-BinA}), for every 
$x\in H^*(X),$ we have 
$$
{\wt i}^*\phi_X(x\otimes 1)={\wt i}^*\pi^*x=\pi_A^*i^*x= \phi_A(i^*x\otimes 1).
$$
For the case $r>0,$ we recall two well-known results:
\begin{itemize}
\item[ 1)] 
If $[E]$ and $[E_A]$ denote the 
cohomology classes of $E$ and $E_A$ in $ \Bl_B X$ and $\wt A,$ 
respectively, then 
\begin{equation}
\label{restriction-exceptionals}
{\wt i}^*[E]=[E_A],
\end{equation}
\item[ 2)] For every $r\geq 0,$ we have 
\begin{equation}
\label{xis}
{\wt j}_*\xi^{r}=[E]^{r+1}\quad{\text{and}}\quad {\wt \ell}_* \xi_A^{r}=[E_A]^{r+1}.
\end{equation}
\end{itemize}
Using the commutativity of the diagrams (\ref{standard-diag}) 
and (\ref{diagram-BinA}), the formulas (\ref{restriction-exceptionals}) 
and (\ref{xis}), and  the projection formula, 
for every $r>0$ and $x\in H^*(X),$ we have 
\begin{align*}
{\wt i}^*\phi_X(x\otimes v^r)=&\,{\wt i}^*({\wt j}_*(f_B^*j^*x\cup\xi^r))\\
=&\,{\wt i}^*({\wt j}_*({\wt j}^*\pi^*x\cup\xi^r))\\
=&\,{\wt i}^*(\pi^*x\cup[E]^{r+1})\\
=&\,({\wt i}^*\pi^*)x\cup[E_A]^{r+1}\\
=&\,({\pi_A}^*i^*)x\cup[E_A]^{r+1}.
\end{align*}
Similarly, we compute 
\begin{align*}
\phi_A(i^*x\otimes v^r)=&\,{\wt \ell}_*(f_A^*\ell^*i^*x\cup\xi_A^r)\\
=&\,{\wt \ell}_*({\wt \ell}^*\pi^*_Ai^*x\cup\xi_A^r)\\
=&\,(\pi_A^*i^*)x\cup[E_A]^{r+1}\\
=&\,{\wt i}^*\phi_X(x\otimes v^r),
\end{align*}
which shows that the diagram (\ref{blowup-formula-BinA}) is  
commutative\footnote{The same argument proves  commutativity 
of the diagram (\ref{blowup-formula-BinA}) with coefficients in $\ZZ$.}.

\smallskip

By \cite[Theorem 3.11]{gitler}, the cohomology of the blow-up is generated 
by the pull-back of classes from the base together with powers of the 
exceptional divisor class, and so the maps $\phi_X$ and $\phi_A$ are 
surjective. Since the restriction map $i^*$ is surjective, from the commutativity 
of diagram (\ref{blowup-formula-BinA}), we find that the map ${\wt i}^*$ 
is surjective, as well. 

To prove that the restriction map 
$
F({\wt i}^*):H^*(F( \Bl_B X))\to H^*(F(\wt A))
$
is surjective, we proceed in similar manner. Once again, the crucial ingredient  
is Theorem 3.11 in \cite{gitler}. As pointed out in \cite[Section 5]{gitler}, this result 
is also valid for the real blow-up, with the degree of the variable $v$ reset to 
one and the first Chern class replaced by the first Stiefel-Whitney class. The 
other arguments are identical with those used for proving the surjectivity of 
${\wt i}^*$  and will therefore not be repeated.

\smallskip

If the $G$-arrangement is effective, then $A$ and $B$ are effective. 
Hence, by Proposition \ref{stretch-effective}, the $G$-pair $(A,B)$ is 
effective. Since in our case, $\wt A$ is the blow-up of $A$ along $B$ 
and  ${\wt A} \cap E$ is the exceptional divisor $E_A=\PP(N_{B|A}),$ 
by Proposition \ref{effectivity-blowup}, $\wt A$ is effective. Furthermore, 
since $\wt A\cap E=\PP(N_{B|A})$ and  $B$ is effective, from Proposition 
\ref{effective-bundles} we conclude that $\wt A\cap E$ is also effective.

Similarly, if the $G$-arrangement is maximal or Galois maximal, then 
$A$ and $B$ are maximal or Galois maximal, respectively. By applying 
Proposition \ref{max-classic} in the maximal case and Proposition 
\ref{derval-GM} in the Galois maximal case, we see that $\wt A$ is 
maximal or Galois maximal, respectively. The maximality or Galois 
maximality of $\wt A\cap E=\PP(N_{B|A})$ is a consequence 
of Proposition \ref{max-bundles}.

\smallskip

{\em Case 4:} 
$\emptyset \subsetneq A\cap B \subsetneq A$ and $B\not\subseteq A.$

\smallskip

In this case, $\wt A$ is the blow-up of $A$ along $A\cap B$ and  
$E_A=\PP(N_{A\cap B|A}).$ Let $g: A\cap B\to A$ and 
$k:A\cap B\to X$ denote the inclusions.

Notice we have a commutative diagram
\begin{equation}
\label{diagram-BcapA}
\xymatrix{ 
&E_A\ar[d]_{f_A}\ar[r]^{{\wt g}}  & \wt A \ar[r]^{\wt i} \ar[d]^{\pi_A} &  \Bl_B X\ar[d]^{\pi}\\
&A\cap B \ar[r]_{\quad g}&  A\ar[r]_i&X,
 }
\end{equation}
where  ${\wt g}:E_A\to {\wt A}$ denotes the inclusion, and, as before, 
$f_A:E_A\to B$ is the restriction of $\pi$ to $E_A.$

As in the previous case, to show that the $G$-pair $(\Bl_B X, \wt A\cap E)$ 
is stretched, it suffices to argue that the $G$-pair $(E, E_A)$ is stretched. 
However, since $A$ and $B$ meet cleanly, $N_{A\cap B|A}$ is a holomorphic 
vector sub-bundle of the restriction of  ${N_{B|X}}$ to ${A\cap B}.$ Hence, 
by applying Lemma \ref{proj-stretch}, it follows that the $G$-pair $(E, E_A)$ 
is indeed stretched.

To show that the pair $( \Bl_B X, {\wt A})$ is stretched we proceed 
as in the previous case and consider the diagram (\ref{blowup-formula-BinA}) 
\begin{equation}
\label{gliter-AcapB}
\xymatrix{ 
&H^*(X)\otimes R \ar[d]_{i^* \otimes Id}\ar[r]^{\quad\phi_X}  & H^*(\Bl_BX)\ar[d]^{{\wt i}^*}\\
& H^*(A)\otimes R  \ar[r]_{\quad\phi'_A}& H^*(\wt A),
}
\end{equation}
where the map $\phi_X$ is defined as in (\ref{gitler-map-X}), while the 
map $\phi'_A$ is given by
 \begin{equation}
 \label{gitler-map-AcapB}
 \phi'_A(a\otimes v^r)=
\begin{cases}
 \pi_A^*a,\, {\text{if}}\, r=0\\
{{\wt g}}_*(f_A^*g^*a\cup\xi_A^r),\, {\text{if}}\, r>0,
\end{cases}
 \end{equation}
where $\xi_A$ the reduction mod $2$ of the first Chern class of 
the tautological line bundle of $E_A,$ respectively. Using the 
commutativity of the diagram (\ref{diagram-BcapA}) and the projection 
formula, for every $r>0$ and $x\in H^*(X),$ we have 
\begin{align*}
\phi'_A(i^*x\otimes v^r)=&\, {\wt g}_*(f_A^*g^*i^*x\cup\xi_A^r)\\
=&\, {\wt g}_*({\wt g}^*\pi_A^*i^*x\cup\xi_A^r),\\
=&\, (\pi_A^*i^*)x\cup[E_A]^{r+1}\\
=&\, {\wt i}^*\phi_X(x\otimes v^r),
\end{align*}
while, as in the previous case,
$$
{\wt i}^*\phi_X(x\otimes 1)={\wt i}^*\pi^*x=\pi_A^*i^*x= \phi_A(i^*x\otimes 1),
$$
proving that the diagram (\ref{gliter-AcapB}) is commutative. By 
\cite[Theorem 3.11]{gitler} the maps $\phi_X$ and $\phi'_A$ are 
surjections. Since the $G$-pair $(X,A)$ is stretched, the map $i^*$ 
is a surjection implying that the map ${\wt i}^*$ is a surjection, as well.

The argument proving that the map 
$
F({\wt i}^*):H^*(F( \Bl_B X))\to H^*(F(\wt A))
$
is surjective coincides with the one in the previous case and will 
not be repeated.

\smallskip

Since the $G$-arrangement $\aa$ is effective, then $A$ and $A\cap B$, 
which both belong to $\aa,$ are effective. As a consequence of 
Proposition \ref{stretch-effective}, the $G$-pair $(A,A\cap B)$ is effective, 
and so, by Proposition \ref{effectivity-blowup}, we find that $\wt A$ is 
effective. In addition, since $A$ and $B$ meet cleanly, then 
$\wt A\cap E=\PP(N_{A\cap B|A}).$ It is an effective $G$-space due 
to the effectivity of $A\cap B$ and Lemma \ref{proj-stretch}.

Similarly, if the $G$-arrangement is maximal or Galois maximal, then 
$A, B$ and $A\cap B$ are maximal or Galois maximal, respectively. 
Since $\wt A$ is the blow-up of $A$ along $A\cap B,$ then $\wt A$ is 
maximal by Proposition \ref{max-classic} or Galois maximal according 
to Proposition \ref{derval-GM}, respectively. The maximality or Galois 
maximality of  $\wt A\cap E=\PP(N_{A\cap B|A})$ follows from 
Proposition \ref{max-bundles}. 
\qed


\begin{thm}
\label{thm-stretched}
Let $\aa$ be a $G$-arrangement of submanifolds in a compact connected complex 
manifold $X,\, \BB\subseteq \aa$ a building set, and $(X_\BB,\aa_\BB)$ its 
wonderful compactification. 
Then:
\begin{itemize}
\item[ 1)] If the arrangement $\aa$ is stretched then the $G$-arrangement 
$(X_\BB,\aa_\BB)$ is stretched.
\item[ 2)] If the  $G$-arrangement $\aa$ is maximal or Galois maximal then 
the $G$-arrangement $(X_\BB,\aa_\BB)$ is maximal or Galois maximal, 
respectively.
\item[ 3)] If the $G$-arrangement $\aa$ is stretched and effective then the 
$G$-arrangement $(X_\BB,\aa_\BB)$ is effective.
\end{itemize}
\end{thm}


\proof
Let $\BB=\{B_1,\dots ,B_N\}$ be a building set of $\aa.$ The proof of 
Theorem  \ref{thm-stretched} is by induction on $N.$ If the arrangement 
$\aa$ is stretched, from Proposition \ref{induction-stretch} we see that 
each time when we blow-up along an element of the building set, 
the resulting $G$-arrangement is stretched. Iterating, the resulting 
$G$-arrangement $(X_\BB,\aa_\BB)$ is stretched, proving the first item.

The second item follows from Proposition \ref{max-classic} and 
Proposition \ref{derval-GM}, respectively, which show that the Smith-Thom 
maximality (Galois maximality) is preserved under blowing-up a 
maximal (Galois maximal) manifold along a smooth maximal 
(Galois maximal) center.

The third item follows by a repeated application of Propositions 
\ref{induction-stretch}.
\qed

As a direct consequence of Theorem \ref{thm-stretched}, we find:


\begin{cor}
\label{w-conj-space}
If the  $G$-arrangement $\aa$ is stretched, maximal and effective, then $X_\BB$ 
is a conjugation space.
\end{cor}


\section{Examples}
\label{examples}


In this section we illustrate the general results obtained above by applying 
them to several natural geometric constructions.


\subsection{Wonderful compactification of subspace arrangements - 
Proof of Theorem \ref{dcp-theorem}}
\label{wonderful-examples}


Let $\PP^n$ be the complex projective space equipped with the standard 
conjugation $\Conj$, and let $\aa$ be an arrangement generated by a finite 
collection $\{L_i\}_{i\in I}$ of $\Conj$-invariant linear subspaces. Then for every 
$i\in I$ the $G$-pair $(\PP^n, L_i)$ is stretched, while $\PP^n$ and each $L_i$ 
is a conjugation space by Theorem \ref{eff-cellular}. In this case, as a direct 
application of Corollary \ref{w-conj-space}, we obtain a proof of Theorem 
\ref{dcp-theorem}.


\subsection{The moduli space of real stable rational curves with marked points - 
Proof of Theorem \ref{R-DM-effective}}
\label{KKspace-effectivity}


The Deligne-Mumford moduli space $\overline{\mathcal M}_{0,n}$  
consists of isomorphism classes of $n$-pointed stable curves of arithmetic 
genus $0$,
$$
((C,J); p_1,\dots,p_n),
$$ 
where $C$ is an underlying tree of smooth spheres, $J$ is a complex 
structure on $C$, and $p_1,\dots, p_n\in C$ are marked points. For each 
element $\sigma$ of the permutation group $S_n$ we have a holomorphic 
bijection$g_\sigma:\overline{\mm}_{0,n}\to \overline{\mm}_{0,n}$, 
$$
g_\sigma((C,J); p_1,\dots,p_n)=((C,J);\sigma(p_1),\dots,\sigma(p_n)),
 $$
and a bijection $\Conj_\sigma: \overline{\mm}_{0,n}\to \overline{\mm}_{0,n},$ 
$$
 \Conj_\sigma((C,J);p_1,\dots,p_n)=((C,-J);\sigma(p_1),\dots,\sigma(p_n)),
$$
which is anti-holomorphic. 
According to Bruno and Mella \cite{bruno}, for every $n\ge 5$ the group of 
holomorphic automorphisms, $\displaystyle \Aut(\overline{\mm}_{0,n})$, 
is composed of $g_\sigma$, $\sigma\in S_n$. Notice that 
$\Conj_{\id}\circ g_\sigma=\Conj_\sigma$ is an involution if and only if 
$\sigma^2=\id.$ This implies that the set of real structures on 
$\overline{\mm}_{0,n}$ is exhausted by the anti-holomorphic involutions 
$\Conj_\sigma$ with $\sigma^2=\id$.

\begin{xpl} \cite{ceyhan}:

\begin{itemize}
\item[ 1)] $\overline{\mathcal M}_{0,3}$ is a point. 
\item[ 2)] $\overline{\mathcal M}_{0,4}$ is biholomorphic to $\CC\PP^1$  with 
the standard real structure, independently of $\sigma$.
\item[ 3)] 
$\overline{\mathcal M}_{0,5}$ is biholomorphic to $\PP^2$ blown-up at four points.
These four points are real if $\sigma=\id$, two of them are real and the other two
are imaginary complex conjugate points, if $\sigma$ is a transposition, and the 
four points form two pairs of imaginary complex conjugate points, 
if $\sigma$ consists of two transpositions.
\end{itemize}
\end{xpl}

\proof[Proof of Theorem \ref{R-DM-effective}]

Kapranov's description \cite{kapranov-chow} of the Deligne-Mumford 
moduli space of rational curves with $n$-marked points exhibits 
$\overline{\mathcal M}_{0,n}$  as the wonderful compactification of 
the arrangement $\aa$ generated by the building set $\BB$ consisting 
of all projective subspaces of $\CC\PP^{n-3}$ spanned by any 
subset of fixed $n-1$ generic points $\Pi=\{p_1,p_2,\dots,p_{n-1}\}.$ 
This description coincides with the wonderful compactification of linear 
subspaces 
$$
A_{n-2}=\{(z_1,\dots, z_{n-1})\in \CC^{n-1}\,|\, z_i=z_j\,,\, 1\leq i<j\leq n-1\}
$$
due to De Concini-Procesi \cite[page 483]{dp} (see also \cite{li}). 

Kapranov's construction depends on the choice of one of the marked points, 
which we take to be the $n$-th marked point.
The moduli space 
$\overline{\mathcal M}_{0,n}$ can be obtained from $\CC\PP^{n-3}$ 
by an iterated blow-up 
$$
\kappa_n:\overline{\mathcal M}_{0,n}\ra \CC\PP^{n-3}
$$
as follows: first, we blow up the points of $\Pi$, next 
the proper images of the lines through the points of $\Pi,$ then the proper 
images of the planes spanned by points of $\Pi,$ etc.

The Kapranov blow-up presentation of $\overline{\mathcal M}_{0,n}$ is 
compatible with the real structure $\Conj_\sigma$  induced by a permutation 
$\sigma\in S_n$ if and only if $\sigma$ fixes the distinguished marked 
point. In this case $\CC\PP^{n-3}$ carries the standard real structure 
$$
\Conj ([z_1:\cdots:z_{n-2}])=[\bar z_1:\cdots:\bar z_{n-2}]
$$
and the points in $\Pi$ are permuted 
accordingly, making the iterated blow-up sequence $\kappa_n$ equivariant at 
each stage.

When $\sigma=\id,$ the $n-1$ points are chosen in 
$F(\CC\PP^{n-3})=\RR\PP^{n-3}.$ Hence, $\BB$ is $G$-invariant, 
and $\aa$ is a $G$-arrangement. Notice that since the elements 
of the buliding set $\BB$ are real linear subspaces in $\CC\PP^{n-3},$ 
equipped with the standard real structure, as in the proof of 
Theorem \ref{dcp-theorem} we see that the $G$-arrangement 
$\aa$ is stretched, maximal and effective. As a direct application 
of Corollary \ref{w-conj-space}, we find that 
$(\overline{\mathcal M}_{0,n}, \Conj_{\id})$ is a conjugation space, 
proving the first part of Theorem \ref{R-DM-effective}.

To prove the second part, notice that if  $\sigma\ne \id$ and the  fixed-point 
set $\Fix \sigma\ne \emptyset,$ the set $\Pi$ in the Kapranov model for 
$\overline{\mathcal M}_{0,n}$ can be chosen $\Conj$-invariant and having 
$\card\vert\Fix\sigma\vert -1$ real points. The associated arrangement of all 
linear subspaces spanned by the subsets of $\Pi$  is merely $\Conj$-invariant. 
Note that at each step in the iterated blow-up $\kappa_n$ the centers are 
pairwise disjoint varieties. As in the proof of Theorem \ref{dcp-theorem}, the 
stretchedness of the blow-up centers is preserved. Blowing up along centers 
with nonempty real locus, the effectivity is preserved due to Theorem 
\ref{thm-stretched}, while for blowing-up a pair of disjoint, complex conjugated 
centers the effectivity is preserved due to Proposition \ref{blowup-emptyset}. 
The Galois maximality of $(\overline{\mathcal M}_{0,n},\Conj_\sigma)$ follows 
from Proposition \ref{derval-GM} applied to Kapranov's iterated blowing-up 
description of the maximal $G$-space $(\CC\PP^{n-3},\Conj)$ along 
Galois maximal centers with possible empty real locus.
\qed


\begin{rmk}
{\rm 
An alternative description of ${\overline \mm}_{0,n}$ is due to Keel and it gives 
an alternative proof of the first part of Theorem \ref{R-DM-effective}.

Keel's construction \cite{keel,li} is the wonderful compactification of the triple
$((\PP^1)^{n-3},\aa,\BB),$ where $\BB$ is the set of all diagonals and augmented diagonals
$$
\Delta_{I,a}=\{(p_4,\dots, p_n)\in (\PP^1)^{n-3},\,\, p_i =a\, \, \forall i\in I\}
$$
for $I\subseteq \{4,\dots,n\},\, |I|\geq 2,$ and $a\in \{0,1,\infty\}.$ Here $\aa$ 
is the set of all intersections of elements in $\BB.$ Again, when  $\sigma=\id$, 
using Lemma \ref{diag-stretched}, one can see that $\aa$ is a stretched, 
maximal $G$-arrangement, and we, as before, we find 
$(\overline{\mathcal M}_{0,n},\Conj_{\id}))$ is a conjugation space 
as a consequence of Corollary \ref{w-conj-space}.
}
\end{rmk}


\begin{rmk}
{\rm 
For the real structure $(\overline{\mathcal M}_{0,n},\Conj_{\sigma})$ with 
$\Fix \sigma=\emptyset,$ an equivariant Kapranov model is unavailable, 
and it will be treated elsewhere.
}
\end{rmk}


\proof[Proof of Corollary \ref{improved-ehkr}]
The proof is a straightforward consequence of Theorems \ref{R-DM-effective} and  
\ref{coh-conj_spaces}.
\qed


\subsection{Wonderful compactifications of configuration spaces - 
Proof of Theorem \ref{eff-wc-config}}
\label{wonderful-configuration}


Let $(X,\Conj)$ be a compact connected complex manifold equipped with an 
anti-holomorphic involution. For $n\ge 2$, a diagonal in $X^n$ is the submanifold
$$
\Delta_I=\{(p_1,\dots,p_n)\in X^n \, |\,  p_i=p_j, \forall\, i,j\in I\}
$$ 
for $I\subseteq \{1,\dots,n\}$, with $|I|\ge 2.$ 
A polydiagonal is an intersection of diagonals
$$
\Delta_{I_1\dots I_k}:=\Delta_{I_1}\cap\cdots\cap\Delta_{I_k}
$$ 
where $I_i\subseteq \{1,\dots,n\}, |I_i|\ge 2$ for each $1\le i\le k.$ The collection $\aa$ 
of all diagonals and polydiagonals is an arrangement of submanifolds in $X^n,$ and the 
configuration space of ordered $n$-tuples of distinct labeled points in $X$ is defined as
$$
\Conf_n(X):=X^n\setminus \bigcup_{|I|\ge 2}\Delta_I=
\{(p_1,\dots,p_n)\in X^n\, |\,  p_i\neq p_j, \forall\, i\neq j\}.
$$
Since all diagonals $\Delta_I$ are $\Conj_n$-invariant with respect to the induced 
anti-holomorphic involution $\Conj_n:X^n\ra X^n$ given by 
$$
\Conj_n(p_1,\dots,p_n)=(\Conj(p_1),\dots,\Conj(p_n)),
$$ 
$\aa$ is a $G$-arrangement, and $\Conf_n(X)$ is a $G$-space.


\begin{lem}
\label{diag-stretched}
For every $1\le i\le k$ and every $I_i\subseteq \{1,\dots,n\}$ with $|I_i|\ge 2,$ the $G$-pair 
$(X^n,\Delta_{I_1\dots I_k})$ is stretched.
\end{lem}


\proof
The proof is a direct consequence of the K\"unneth formula.
\qed

The space $\Conf_n(X)$ admits several wonderful 
$G$-compactifications\footnote{We refer the interested reader to \cite{li} for details.}.

\begin{itemize}

\item[ 1)] An example is the Fulton-MacPherson compactification \cite{fm}, denoted 
by $X[n].$ In this case, the building set $\BB$ is the set of all diagonals $\Delta_I,$ for $
I\subseteq \{1,\dots,n\}$, with $|I|\ge 2.$

\item[ 2)] Ulyanov's compactification  \cite{ulyanov}, denoted by $X\langle n\rangle,$ 
is another wonderful $G$-compactification of $\Conf_n(X).$ In this case, the building set 
$\BB$ is the set of all polydiagonals.

\item[ 3)] The Kuperberg-Thurston compactification $X^{\Gamma},$ 
where $\Gamma$ is a connected graph with $n$ labeled vertices 
\cite{li}. In this case, $\BB$ is  an appropriate set of polydiagonals.

\end{itemize}

The proof of Theorem \ref{eff-wc-config} is now a direct consequence of 
Lemma \ref{diag-stretched} and Theorem \ref{thm-stretched}.


\section{Hilbert Squares}


In this section we apply Kalinin effectivity to the study of Hilbert squares 
of real algebraic varieties and obtain a formula for the Smith-Thom deficiency 
of $X^{[2]}$.

\medskip

Let $(X,\Conj)$ be a compact $n$-dimensional complex manifold equipped 
with a real structure. We denote by $F$ the fixed locus of $\Conj,$ and let 
$F_i,$ $i=1,\dots,r$ be the connected components of $F.$ The involution 
$\tau:X\times X\to X\times X$ permuting the factors lifts to an involution 
$\Bl(\tau)$ on the blowup $\Bl_\Delta (X\times X)$ of  $X\times X$ along the 
diagonal $\Delta\subset X\times X.$ The quotient of $\Bl_\Delta (X\times X)$ 
by $\Bl(\tau)$ is then naturally isomorphic to the Hilbert square $X^{[2]}.$ The 
real structure $\Conj$ on $X$ naturally lifts to a real structure on $X^{[2]},$ denoted 
by $\Conj^{[2]}.$

By construction, the branch locus $E\subset X^{[2]}$ of the double 
ramified covering $\Bl_\Delta (X\times X)\to X^{[2]}$ is naturally 
isomorphic to $\PP(T^*X),$ it is $\Conj^{[2]}$-invariant and coincides with 
the exceptional divisor of the canonical projection $X^{[2]}\to X^{(2)}$ 
to the symmetric square $X^{(2)}$ of $X.$ Notice ({\it cf.} \cite{loss-general})
that its fixed locus 
$F(E)=\PP_\RR(T^*F)$ is simultaneously the boundary for two real 
$2n$-manifolds, $\HH_0$ and $\HH,$ 
whose interiors satisfy
$$
\interior \HH_{0}\cong (X/\Conj) \setminus F,
\quad  \interior \HH\cong F^{(2)}\setminus \Delta F,
$$
where $\Delta F$ is the diagonal embedding of $F$ in $F^{(2)}.$ Let
$$
\inc_0: F(E)\ra \HH_0 \quad{\text{and}}\quad \inc_1: F(E)\ra \HH
$$
be the inclusion maps, and 
$\displaystyle \mu =(\inc_{0},\inc_1) : F(E)\ra \HH_0\sqcup \HH.$ 

Recall a well known statement, dual to the one proved in 
\cite[Lemma 2.10]{loss-general}.


\begin{lemma}
\label{half-and-half}
Let $M$ be a compact $n$-dimensional manifold with boundary, 
and $\inc: \partial M\to M$ the inclusion. Then 
$$
 \Dim \coker (\inc^*) =\rank (\inc^*)=\frac12 \Dim \, H^*(\partial M). 
 $$
 \qed
\end{lemma}


\begin{prop}
\label{eff-images}
If $X$ is Galois maximal and effective, then $\im(\inc_0^*)=\im(\inc_1^*).$
\end{prop}


\proof
By Lemma \ref{half-and-half}, it is sufficient to show that 
$\im(\inc_1^*)\cap \im(\inc_0^*)$ contains a subspace of dimension
$\frac{n}2\beta_*(F)$.

\smallskip

Let us consider a decomposition $X=U\cup V$ where $U$ is 
a closed
equivariant tubular neighborhood of $F$ and $V$ is the closure of 
$X\setminus U$. Here, $U$ can be identified with a disc subbundle $
D(F)$ of $T^* F$ and $U\cap V$ with a sphere subbundle 
$S(F)=\partial D(F)=\partial V$. By applying the Mayer-Vietoris sequence 
in equivariant cohomology to this decomposition, we get a long exact 
sequence
$$
\dots \to H^n_G(X)\to H_G^n(U)\oplus H_G^n(V)\to H_G^n(S)\to \dots 
$$
Since $U$ is homotopy equivalent to $F$ and $G$ acts freely on 
$V$ and $S$, this sequence turns into
\begin{align*}
\dots \to H^n_G(X)\to \left(\bigoplus_p H^p(F)\otimes 
H^{n-p}(\RR\PP^\infty)\right) &\oplus H^n(V/\Conj)\\
\to &\, H^n(F(E))=H^n(S/\Conj )\to \dots
\end{align*}

According to Proposition \ref{Image-Beta}, due to effectivity of $X,$ 
for every $z\in H^p(F), p\ge 0,$ there exists $\hat z\in H^{2p}_G(X)$ 
such that the first component of
the mapping 
$$
H^*_G(X)\to (H^*(F)\otimes H^*(\RR\PP^\infty))\oplus H^*(V/\Conj )
$$
sends $\hat z$ to $u^p z+ u^{p-1}\Sq^1 z +\dots + \Sq^p z,$
where $u$ is the characteristic class of the tautological line bundle 
over $BG=\RR \PP^\infty$.

Let now $w\in H^*(V/\Conj )$ denote the second summand of the image of 
$\hat z$. Due to exactness of the Mayer-Vietoris sequence, the image of
$$
(u^p z + u^{p-1}\Sq^1 z  + \dots + \Sq^p z, w)
$$ 
in $H^*(F(E))=H^*(S/\Conj )$ is equal to $0$. This gives
$$
\gamma^{p} z +\gamma^{p-1}\Sq^1 z+\dots 
+\gamma\Sq^{p-1} z +\Sq^{p} z = \inc_0^*w 
$$
where $\gamma\in H^1(F(E))$ is the characteristic class of the 
tautological line bundle over $E$.

Due to Proposition \ref{sym-kalinin}, $(F\times F, \tau)$ is also a Galois 
maximal effective $G$-space. Therefore, the above arguments apply 
literally to it and show that for every $z\in H^p(F)$ there exists
$w'\in H^*(F^{(2)}\setminus \Delta F),$ such that
$$
\gamma^{p} z +\gamma^{p-1}\Sq^1 z+\dots 
+\gamma\Sq^{p-1} z +\Sq^{p} z = \inc_1^*w'.
$$

Finally, since, for any integer $l\ge 1$, we have
$$
\inc_0^*\gamma_0^l w= \gamma^l \inc_0^* w \quad{\text{and}}\quad 
 \inc_1^*\gamma_1^l w'= \gamma^l \inc_1^* w',
$$
where $\gamma_{0}$ and $\gamma_1$ are characteristic 
classes of the corresponding 2-sheeted coverings,
we conclude that $\im(\inc_1^*)\cap \im(\inc_0^*)$ contains a 
subspace generated by
$$
\gamma^l \left(\gamma^{p} z +\gamma^{p-1}\Sq^1 z+\dots 
+\gamma\Sq^{p-1} z +\Sq^{p} z\right), \, 0\le l\le n-1-p, z\in H^p(F).
$$
Thus, there remains to notice that these generators are linear 
independent, and therefore the dimension of this subspace is equal to
$$
\sum_{p=0}^{n-1} (n-p)\beta_p(F)=\frac{n}2\beta_*(F).
$$
\qed

Let $\delta_k=\rank \Delta_k, \, k\geq 0$ where $\Delta_k$ are the 
connecting homomorphisms in the Smith exact sequence (\ref{rses}).


\begin{prop}
\label{6.1-deficiency}
Let $X$ be a real nonsingular  $n$-dimensional  algebraic variety 
with $\Tors_2H_*(X,\ZZ)=0$ and Smith-Thom deficiency $\defi(X)=a$.
Then the Smith-Thom deficiency of $X^{[2]}$ is given by
$$
\defi(X^{[2]})=2\,\rank \mu_*+\sum_{k=1}^{2n}(2k-1)\delta_k
+a\beta_*+\frac{a(a-1)}2-n\beta_*(F(X))-\beta_{\rm odd}.
$$
In particular, if $X$ is maximal then
$
\displaystyle
\defi(X^{[2]})=2\,\rank \mu_*-n\beta_*-\beta_{\rm odd}.
$
\end{prop}


\proof
This computation repeats almost literally the proof of 
\cite[Proposition 6.1]{loss-general}. The changes are caused by 
replacing the maximality assumption by Galois maximality, and are 
the following. The equality $\beta_*(F(E))=n\beta_*$ is to be 
replaced by $\beta_*(F(E))=n\beta_*(F(X))$.  In particular, the equality 
$\displaystyle \beta_*(\HH)=
\frac12\sum_{i=1}^r\beta_*^2(F_i)+\frac{n-1}{2}\beta_*$ is replaced by  
$\displaystyle \beta_*(\HH)=\frac12\sum_{i=1}^r\beta_*^2(F_i)+
\frac{n-1}{2}\beta_*(F(X))$. By contrast, the equality 
$\displaystyle \beta_*(\HH_0)=\frac{n}{2}\beta_*,$ proved in  
 \cite[Lemma 3.1]{loss-general} as an application of  the 
Smith exact sequence, is to be replaced by 
$\displaystyle \beta_*(\HH_0)=n\beta_* -\frac{n}{2}\beta_*(F(X)) -
\sum^{2n}_{k=1}(2k-1)\delta_k.$
\qed


\begin{lemma}\label{rank-mu}
If $\im (\inc_{0}^*)=\im (\inc_{1}^*)$ then 
$\rank \mu_*=\frac{n}2 \beta_*(F(X))$
\end{lemma}


\proof
Since $\beta_*(E)=n\beta_*(F(X))$, the relation stated follows 
from Lemma \ref{half-and-half} applied to any of the halves $U, V$.
\qed


\proof[Proof of Theorem \ref{defect-GM}] 
Due to Proposition \ref{eff-images} we have $\im (\inc_{0}^*)=\im (\inc_{1}^*).$ 
Therefore, $\rank \mu_*=\dim \im (\inc_{0}^*)=\dim \im (\inc_{1}^*)$
so that by Lemma \ref{rank-mu} we get $\rank \mu_* =\frac{n}2 \beta_*(F(X)).$ 
This, together with Proposition \ref{6.1-deficiency}, completes the proof.
\qed


\proof[Proof of Corollary \ref{square}] 
Since $X$ is a conjugation space, it is maximal and effective. In particular, 
to compute the Smith-Thom deficiency of $X^{[2]}$ one can apply Theorem 
\ref{defect-GM} and we notice that  for conjugation spaces we have 
$\beta_{\rm odd}=0,\, \delta_k=0$ and $\defi(X)=0.$ Therefore 
$\defi(X^{[2]})=0.$ 
\qed


\bibliographystyle{alpha}

\begin{thebibliography}{7}



\bibitem{AP}
{\sc C.~Allday, V.~Puppe,}
{\em  Cohomological methods in transformation groups.}
 Cambridge Studies in Advanced
Mathematics, vol. 32. Cambridge University Press, Cambridge, 1993.



\bibitem{bruno}
{\sc A. Bruno, M. Mella,} 
{\em The automorphism group of $\overline M_{0,n}$.} 
J. Eur. Math. Soc. (JEMS) 15 (2013), no. 3, 949–968.



\bibitem{ceyhan}
{\sc \"O. Ceyhan,} 
{\em Graph homology of the moduli space of pointed real curves of genus zero.}
Selecta Math. (N.S.) 13 (2007), no. 2, 203–237.


\bibitem{cgz}
{\sc X. Chen, P. Georgieva, and A. Zinger,} 
{\em The cohomology ring of the Deligne-Mumford space of real 
rational curves with conjugate marked points,} 
\href{https://doi.org/10.48550/arXiv.2305.08798}{arXiv:2305.08798 [math.AG]}.


\bibitem{dp}
{\sc C.~De Concini, C.~Procesi, }
{\em Wonderful models of subspace arrangements.} 
Selecta Mathematica 1 (1995), 459--494.


\bibitem{dik}
{\sc A. Degtyarev, I. Itenberg, V. Kharlamov,} 
{\em Real Enriques surfaces.} 
Lecture Notes in Mathematics, 1746. Springer-Verlag, Berlin, 2000.


\bibitem{dk}
{\sc A. Degtyarev, V. Kharlamov,} 
{\em Topological properties of real algebraic varieties: Rokhlin's way.} 
Uspekhi Mat. Nauk 55 (2000), no. 4 (334), 129--212; 
translation in Russian Math. Surveys 55 (2000), no. 4, 735--814.



\bibitem{dk-halves}
{\sc A. Degtyarev, V. Kharlamov,}  
{\em Halves of a real Enriques surface. } 
Comm. Math. Helv., 71, no. 4 (1996), 628--663.
Extended version: {\em Distribution of the components of a real Enriques surface.} 
Preprint of the Max-Planck Institute, MPI/95-58 (1995)


\bibitem{derval}
{\sc D. Derval,}
{\em Vari\'et\'es Galois-maximales et \'eclatements.}
Manuscripta Math.   110 (2003), no. 4, 467--474.


\bibitem{ehkr}
{\sc P. Etingof, A. Henriques, J. Kamnitzer, E. Rains,} 
{\em The cohomology ring of the real locus of the moduli space of stable curves 
of genus 0 with marked points.} 
 Ann. of Math. 171 (2010), no. 2, 731--777.


\bibitem{fm}
{\sc W. Fulton, R. MacPherson,} 
{\em A compactification of configuration spaces.} Ann. of Math. (2) 139 (1994), no. 1, 183--225.



\bibitem{fp}
{\sc M. Franz, V. Puppe,} 
{\em Steenrod squares on conjugation spaces.} 
C. R. Acad. Sci. Paris, Ser. I 342 (2006), 187 -- 190.


\bibitem{Fr}
{\sc M. Franz,} 
{\em Symmetric products of equivariantly formal spaces.} 
Canad. Math. Bull. 61 (2018), no. 2, 272--281.


\bibitem{gitler}
{\sc S. Gitler,}
{\em The cohomology of blow ups.} 
Bol. Soc. Mat. Mexicana (2) 37 (1992), no. 1-2, 167 -- 175.



\bibitem{HHP}
{\sc J.-C. Hausmann, T. Holm, V. Puppe,}
{\em  Conjugation spaces.}
 Alg. Geom. Topology 5 (2005), 923–964.


\bibitem{kalinin-ss}
{\sc I.O.~Kalinin,} 
{\em Cohomological characteristics of real projective hypersurfaces.} 
Algebra i Analiz 3 (1991), no. 2, 91--110 (Russian); 
English transl. in St. Petersburg Math. J. 3 (1992), no. 2, 313--332.



\bibitem{kalinin-eff}
{\sc I.O.~Kalinin,} 
{\em Cohomological characteristics of real algebraic manifolds.} 
Algebra i Analiz 14(2002), no.5, 39–72 (Russian); translation in
St. Petersburg Math. J. 14 (2003), no. 5, 739–763.



\bibitem{kalinin}
{\sc I.O.~Kalinin,} 
{\em Cohomology of real algebraic varieties.} (Russian) 
Zap. Nauchn. Sem. S.-Peterburg. Otdel. Mat. Inst. Steklov. (POMI) 299 (2003), 
Geom. i Topol. 8, 112–151, 328--329; translation in J. Math. Sci. (N.Y.) 131 (2005), 
no. 1, 5323--5344.


\bibitem{kapranov-chow}
{\sc M. M. Kapranov,} 
{\em Chow quotients of Grassmannians, I.}
Adv. Soviet Math. 16, Part 2, Amer. Math. Soc., Providence, 1993, 29 – 110.

\bibitem{keel}
{\sc S. Keel,}
{\em Intersection theory of moduli space of stable n-pointed curves of genus zero,} 
Trans. Amer. Math. Soc. 330 (1992), 545--574.


\bibitem{loss-surfaces}
{\sc V. Kharlamov, R. R\u asdeaconu,}
 {\em On the maximality problem for Hilbert squares of real surfaces.}
Selecta Math., New Ser. 31, 104 (2025). 


\bibitem{loss-general}
{\sc V. Kharlamov, R. R\u asdeaconu,}
{\em On the Smith-Thom deficiency of Hilbert squares.} 
J. Topol.   17 (2024), no. 2, Paper No. e12345, 29 pp.


\bibitem{krasnov}
{\sc V.A. Krasnov,}
{\em Harnack-Thom inequalities for mappings of real algebraic varieties.} 
Izv. Akad. Nauk SSSR, Ser. Mat. 47 (2) (1983) 268--297.


\bibitem{kt}
{\sc G. Kuperberg, D. Thurston,}
{\em Perturbative 3-manifold invariants by cut-and-paste topology.} 
preprint, math.GT/9912167.


\bibitem{li}
{\sc L. Li,} 
{\em Wonderful compactification of an arrangement of subvarieties.} 
Michigan Math. J. 58 (2009), no. 2, 535--563.


\bibitem{rains}
{\sc E. Rains,}
{\em  The homology of real subspace arrangements,} 
J. Topol. 3 (2010), no. 4, 786--818.


\bibitem{ulyanov}
{\sc A. P. Ulyanov,}
{\em Polydiagonal compactification of configuration spaces.}
J. Algebraic Geom.   11 (2002), no. 1, 129--159.


\bibitem{hamel}
{\sc J. van Hamel,}
{\em  Geometric cohomology frames on Hausmann-Holm-Puppe conjugation spaces.}
Proc. Amer. Math. Soc. 135 (2007), no. 5, 1557--1564.


\bibitem{voisin}
{\sc C. Voisin,}
{\em  Hodge Theory and Algebraic Geometry, I.}
Cambridge studies in advanced mathematics, Vol. 76, 2002.



\end{thebibliography}

\end{document}